\documentclass[11pt, reqno, psamsfonts]{amsart}
\pdfoutput=1

\usepackage{amssymb}
\usepackage{amsthm}
\usepackage{dsfont}
\usepackage{amsmath}
\usepackage{latexsym}
\usepackage[T1]{fontenc}
\usepackage[utf8]{inputenc}
\usepackage[russian, french, english]{babel}

\usepackage{graphicx}
\usepackage{wrapfig}
\usepackage[justification=centering, labelfont=bf]{caption}
\usepackage{mathtools}
\usepackage{tikz}
\usepackage{amsbsy}
\usepackage[inline]{enumitem}
\usepackage{mathrsfs}
\usetikzlibrary{shapes,snakes}
\usetikzlibrary{arrows.meta}
\usetikzlibrary{decorations.pathmorphing}
\usetikzlibrary{patterns}
\usepackage{float}
\usepackage{array}
\usepackage{multicol}
\usepackage{stmaryrd}
\usepackage{cancel} 
\usepackage{lmodern}
\usepackage{upgreek}
\usepackage{titlesec}
\usepackage[spacing=true,kerning=true,babel=true,tracking=true]{microtype}
\usepackage[shortcuts]{extdash}
\usepackage[foot]{amsaddr}
\usepackage[left=1in,right=1in,top=1in,bottom=1in,bindingoffset=0cm]{geometry}
\usepackage{bm}
\usepackage{centernot}
\usepackage{mdframed}
\usepackage[hidelinks]{hyperref}
\usepackage{wasysym}

\usepackage[backend=biber, style=alphabetic, sorting=nyt, maxnames=100,backref=true]{biblatex}
\title{Coloring graphs with forbidden bipartite subgraphs}
\date{}
\author{\lsstyle James~Anderson}
\email{janderson338@math.gatech.edu}
\author{\lsstyle Anton~Bernshteyn}
\email{bahtoh@gatech.edu}
\author{\lsstyle Abhishek~Dhawan}
\email{abhishek.dhawan@math.gatech.edu}
\address{\textls{\normalfont{}School of Mathematics, Georgia Institute of Technology, Atlanta, GA, USA}}
\thanks{Research of the second named author was partially supported by the NSF grant DMS-2045412.}

\newtheoremstyle{bfnote}%
{}{}%
{\slshape}{}%
{\bfseries}{\bfseries.}%
{ }%
{\thmname{#1}\thmnumber{ #2}\thmnote{ \ep{\normalfont{}#3}}}

\theoremstyle{bfnote}
\newtheorem{theo}[equation]{Theorem}
\newtheorem*{theo*}{Theorem}
\newtheorem{prop}[equation]{Proposition}
\newtheorem{Lemma}[equation]{Lemma}
\newtheorem{claim}[equation]{Claim}
\newtheorem{corl}[equation]{Corollary}
\newtheorem{conj}[equation]{Conjecture}

\newtheorem*{corl*}{Corollary}

\newcounter{ForClaims}[section]
\newtheorem{subclaim}{Subclaim}[ForClaims]

\theoremstyle{definition}
\newtheorem{defn}[equation]{Definition}
\newtheorem*{defn*}{Definition}

\newtheorem*{exmp*}{Example}

\theoremstyle{remark}
\newtheorem*{ques*}{Question}
\newtheorem*{remk*}{Remark}

\newcommand*{\myproofname}{Proof}
\newenvironment{claimproof}[1][\myproofname]{\begin{proof}[#1]}{\end{proof}}

\newcommand{\0}{\emptyset}
\newcommand{\set}[1]{\{#1\}}
\newcommand{\N}{{\mathbb{N}}}

\renewcommand{\P}{\mathbb{P}}
\newcommand{\E}{\mathbb{E}}
\renewcommand{\epsilon}{\varepsilon}
\newcommand{\eps}{\epsilon}
\renewcommand{\phi}{\varphi}
\renewcommand{\theta}{\vartheta}
\renewcommand{\leq}{\leqslant}
\renewcommand{\geq}{\geqslant}
\newcommand{\defeq}{\coloneqq}
\newcommand{\im}{\mathrm{im}}
\newcommand{\bemph}[1]{{\normalfont#1}} 
\newcommand{\ep}[1]{\bemph{(}#1\bemph{)}} 

\newcommand{\emphdef}[1]{\textbf{\textit{{#1}}}}
\newcommand{\pto}{\dashrightarrow}
\newcommand{\emphd}[1]{\emphdef{#1}}
\newcommand{\keep}{\mathsf{keep}}
\newcommand{\uncolor}{\mathsf{uncolor}}
\newcommand{\blank}{\mathsf{blank}}
\newcommand{\col}{\mathsf{col}}
\newcommand{\LLL}{\text{Lov\'asz Local Lemma}}
\numberwithin{equation}{section}
\newcommand{\dom}{\mathrm{dom}}
\newcommand{\Bad}{\mathsf{Bad}}
\newcommand{\Sad}{\mathsf{Sad}}
\newcommand{\Happy}{\mathsf{Happy}}
\newcommand{\Good}{\mathsf{Good}}

\titleformat{\section}[block]{\scshape}{\thesection.}{1ex}{}
\titleformat{\subsection}[block]{\bfseries}{\thesubsection.}{1ex}{}
\titleformat{\subsubsection}[runin]{\itshape}{\bfseries\upshape\thesubsubsection.}{1ex}{}[.---]

\titlespacing*{\section}{0pt}{*3}{*1}
\titlespacing*{\subsection}{0pt}{*3}{*1}
\titlespacing*{\subsubsection}{0pt}{*1.5}{*0}

\renewbibmacro{in:}{}

\renewbibmacro*{volume+number+eid}{%
	\printfield{volume}%
	\setunit*{\addnbspace}
	\printfield{number}%
	\setunit{\addcomma\space}%
	\printfield{eid}}

\DeclareFieldFormat[article]{volume}{\textbf{#1}\space}
\DeclareFieldFormat[article]{number}{\mkbibparens{#1}}

\DeclareFieldFormat{journaltitle}{#1,}
\DeclareFieldFormat[thesis]{title}{\mkbibemph{#1}\addperiod}
\DeclareFieldFormat[article, unpublished, thesis]{title}{\mkbibemph{#1},}
\DeclareFieldFormat[book]{title}{\mkbibemph{#1}\addperiod}
\DeclareFieldFormat[unpublished]{howpublished}{#1, }

\DeclareFieldFormat{pages}{#1}

\DeclareFieldFormat[article]{series}{Ser.~#1\addcomma}

\setlist{topsep=3pt,itemsep=3pt}

\begin{document}

\vspace*{0pt}

\maketitle
\begin{abstract}
     A conjecture of Alon, Krivelevich, and Sudakov states that, for any graph $F$, there is a constant $c_F > 0$ such that if $G$ is an $F$-free graph of maximum degree $\Delta$, then $\chi(G) \leq c_F \Delta / \log\Delta$. Alon, Krivelevich, and Sudakov verified this conjecture for a class of graphs $F$ that includes all bipartite graphs. Moreover, it follows from recent work by Davies, Kang, Pirot, and Sereni that 
     if $G$ is $K_{t,t}$-free, then $\chi(G) \leq (t + o(1)) \Delta / \log\Delta$ as $\Delta \to \infty$. We improve this bound to $(1+o(1)) \Delta/\log \Delta$, making the constant factor independent of $t$. We further extend our result to the DP-coloring setting (also known as correspondence coloring), introduced by Dvo\v{r}\'ak and Postle. 
    %
\end{abstract}

\noindent

\section{Introduction}

    All graphs in this paper are finite, undirected, and simple. The starting point of our investigation is the following celebrated conjecture of Alon, Krivelevich, and Sudakov:
    
    \begin{conj}[{Alon--Krivelevich--Sudakov \cite[Conjecture 3.1]{AKSConjecture}}]\label{conj:AKS}
        For every graph $F$, there is a constant $c_F > 0$ such that if $G$ is an $F$-free graph of maximum degree $\Delta \geq 2$, then $\chi(G) \leq c_F\Delta/\log \Delta$. 
    \end{conj}
    
    Here we say that $G$ is \emphdef{$F$-free} if $G$ has no subgraph \ep{not necessarily induced} isomorphic to $F$. As long as $F$ contains a cycle, the bound in Conjecture~\ref{conj:AKS} is best possible up to the value of $c_F$, since there exist $\Delta$-regular graphs $G$ of arbitrarily high girth with $\chi(G) \geq (1/2)\Delta/\log \Delta$ \cite{BollobasIndependence}. On the other hand, the best known general upper bound is $\chi(G) \leq c_F \Delta \log \log \Delta/\log \Delta$ due to Johansson \cite{Joh_sparse} \ep{see also \cite{Molloy}}, which exceeds the conjectured value by a $\log \log \Delta$ factor.
    
    Nevertheless, there are some graphs $F$ for which Conjecture~\ref{conj:AKS} has been verified. Among the earliest results along these lines is the theorem of Kim \cite{Kim95} that if $G$ has girth at least $5$ \ep{that is, $G$ is $\set{K_3, C_4}$-free}, then $\chi(G) \leq (1 + o(1))\Delta/\log \Delta$. \ep{Here and in what follows $o(1)$ indicates a function of $\Delta$ that approaches $0$ as $\Delta \to \infty$.} Johansson \cite{Joh_triangle} proved Conjecture \ref{conj:AKS} for $F = K_3$; that is, Johansson showed that if $G$ is triangle-free, then $\chi(G) \leq c\Delta/\log \Delta$ for some constant $c > 0$. Johansson's proof gave the value $c = 9$ \cite[125]{MolloyReed}, which was later improved to $4 + o(1)$ by Pettie and Su \cite{PS15} and, finally, to $1+o(1)$ by Molloy \cite{Molloy}, matching Kim's bound for graphs of girth at least $5$.
    
    In the same paper where they stated Conjecture~\ref{conj:AKS}, Alon, Krivelevich, and Sudakov verified it for the complete tripartite graph $F = K_{1,t,t}$ \cite[Corollary 2.4]{AKSConjecture}. \ep{Note that the case $t = 1$ yields Johansson's theorem.}  Their results give the bound $c_F = O(t)$ for such $F$, which was recently improved to $t + o(1)$ by Davies, Kang, Pirot, and Sereni \cite[\S5.6]{DKPS}. Numerous other results related to Conjecture~\ref{conj:AKS} can be found in the same paper.
    
    Here we are interested in the case when the forbidden graph $F$ is bipartite. It follows from the result of Davies, Kang, Pirot, and Sereni mentioned above that if $F = K_{t,t}$, then Conjecture~\ref{conj:AKS} holds with $c_F = t + o(1)$. Prior to this work, this has been the best known bound for all $t \geq 3$ \ep{the graph $F = K_{2,2}$ satisfies Conjecture~\ref{conj:AKS} with $c_F = 1 + o(1)$ by \cite[Theorem~4]{DKPS}}. We improve this bound to $1 + o(1)$ for all $t$ \ep{so only the lower order term actually depends on the graph $F$}:
    
    \begin{theo}\label{theo:col}
        For every bipartite graph $F$ and every $\eps > 0$, there is $\Delta_0 \in \N$ such that every $F$-free graph $G$ of maximum degree $\Delta \geq \Delta_0$ satisfies $\chi(G) \leq (1+\eps)\Delta/\log \Delta$.
    \end{theo}
    
    As witnessed by random $\Delta$-regular graphs, the upper bound on $\chi(G)$ in Theorem~\ref{theo:col} is asymptotically optimal up to a factor of $2$ \cite{BollobasIndependence}. Furthermore, this bound coincides with the so-called {shattering threshold} for colorings of random graphs of average degree $\Delta$ \cite{Zdeborova,Achlioptas}, as well as the density threshold for factor of i.i.d.~independent sets in $\Delta$-regular trees \cite{RV}, which suggests that reducing the number of colors further would be a challenging problem, even for graphs $G$ of large girth. Indeed, it is not even known if such graphs admit independent sets of size greater than $(1+o(1))|V(G)| \log \Delta/\Delta$.

    In view of the results in \cite{AKSConjecture} and \cite{DKPS}, it is natural to ask if a version of Theorem~\ref{theo:col} also holds for $F = K_{1,t,t}$. We give the affirmative answer in a forthcoming paper \cite{nbhd}, where we use some of the techniques developed here to prove that every $K_{1,t,t}$-free graph $G$ satisfies $\chi(G) \leq (4+o(1))\Delta/\log \Delta$. In other words, we eliminate the dependence on $t$ in the constant factor, although we are unable to reduce it all the way to $1 + o(1)$.
     
    Returning to the case of bipartite $F$, we establish an extension of Theorem~\ref{theo:col} in the context of \emph{DP-coloring} \ep{also known as \emph{correspondence coloring}}, which was introduced a few years ago by Dvo\v{r}\'ak and Postle \cite{DPCol}. DP-coloring is a generalization of list coloring. Just as in ordinary list coloring, we assume that every vertex $u \in V(G)$ of a graph $G$ is given a list $L(u)$ of colors to choose from. In contrast to list coloring though, the identifications between the colors in the lists are allowed to vary from edge to edge. That is, each edge $uv \in E(G)$ is assigned a matching $M_{uv}$ (not necessarily perfect and possibly empty) from $L(u)$ to $L(v)$. A \emph{proper DP-coloring} then is a mapping $\phi$ that assigns a color $\phi(u) \in L(u)$ to each vertex $u \in V(G)$ so that whenever $uv \in E(G)$, we have $\phi(u)\phi(v) \notin M_{uv}$. Note that list coloring is indeed a special case of DP-coloring which occurs when the colors ``correspond to themselves,'' i.e., for each $c \in L(u)$ and $c' \in L(v)$, we have $cc' \in M_{uv}$ if and only if $c = c'$.
    
    Formally, we describe DP-coloring using an auxiliary graph $H$ called a \emph{DP-cover} of $G$. Here we treat the lists of colors assigned to distinct vertices as pairwise disjoint \ep{this is a convenient assumption that does not restrict the generality of the model}. The definition below is a modified version of the one given in \cite{JMTheorem}: 

    \begin{defn}\label{corrcov}
    A \emphdef{DP-cover} \ep{or a \emphd{correspondence cover}} of a graph $G$ is a pair $\mathcal{H} = (L, H)$, where $H$ is a graph and $L \colon V(G) \to 2^{V(H)}$ is a function such that:
    \begin{itemize}
        \item The set $\{L(v) \,:\, v \in V(G)\}$ forms a partition of $V(H)$.
        \item For each $v \in V(G)$, $L(v)$ is an independent set in $H$.
        \item For $u$, $v \in V(G)$, the induced subgraph $H[L(u) \cup L(v)]$ is a matching; this matching is empty whenever $uv \notin E(G).$
    \end{itemize}
    We refer to the vertices of $H$ as \emphd{colors}. 
    For $c \in V(H)$, we let $L^{-1}(c)$ denote the \emphd{underlying vertex} of $c$ in $G$, i.e., the unique vertex $v \in V(G)$ such that $c \in L(v)$. If two colors $c$, $c' \in V(H)$ are adjacent in $H$, we say that they \emphd{correspond} to each other and write $c \sim c'$.
    An \emphd{$\mathcal{H}$-coloring} is a mapping $\phi \colon V(G) \to V(H)$ such that $\phi(u) \in L(u)$ for all $u \in V(G)$. Similarly, a \emphd{partial $\mathcal{H}$-coloring} is a partial map $\phi \colon V(G) \pto V(H)$ such that $\phi(u) \in L(u)$ whenever $\phi(u)$ is defined. A \ep{partial} $\mathcal{H}$-coloring $\phi$ is \emphd{proper} if the image of $\phi$ is an independent set in $H$, i.e., if $\phi(u) \not \sim \phi(v)$ 
    for all $u$, $v \in V(G)$ such that $\phi(u)$ and $\phi(v)$ are both defined. 
    A DP-cover $\mathcal{H}$ is \emphdef{$k$-fold} for some $k \in \N$ if $|L(u)| \geq k$ for all $u \in V(G)$. The \emphdef{DP-chromatic number} of $G$, denoted by $\chi_{DP}(G)$, is the smallest $k$ such that $G$ admits a proper $\mathcal{H}$-coloring with respect to every $k$-fold DP-cover $\mathcal{H}$.
    \end{defn}
    
    An interesting feature of DP-coloring is that it allows one to put structural constraints not on the base graph, but on the cover graph instead. For instance, Cambie and Kang \cite{CK} made the following conjecture:
    
    \begin{conj}[{Cambie--Kang \cite[Conjecture 4]{CK}}]\label{conj:CK}
        For every $\eps > 0$, there is $d_0 \in \N$ such that the following holds. Let $G$ be a triangle-free graph and let $\mathcal{H} = (L,H)$ be a DP-cover of $G$. If $H$ has maximum degree $d \geq d_0$ and $|L(u)| \geq (1+\eps)d/\log d$ for all $u \in V(G)$, then $G$ admits a proper $\mathcal{H}$-coloring.
    \end{conj}
    
    The conclusion of Conjecture~\ref{conj:CK} is known to hold if $d$ is taken to be the maximum degree of $G$ rather than of $H$ \cite{JMTheorem} \ep{notice that $\Delta(G) \geq \Delta(H)$, so a bound on $\Delta(G)$ is a stronger assumption than a bound on $\Delta(H)$}. Cambie and Kang \cite[Corollary 3]{CK} verified Conjecture~\ref{conj:CK} when $G$ is not just triangle-free but bipartite. Amini and Reed \cite{AminiReed} and, independently, Alon and Assadi \cite[Proposition 3.2]{Pallette} proved a version of Conjecture~\ref{conj:CK} for list coloring, but with $1 + o(1)$ replaced by a larger constant \ep{$8$ in \cite{Pallette}}. To the best of our knowledge, it is an open problem to reduce the constant factor to $1 + o(1)$ even in the list coloring framework.
    
    Notice that in Cambie and Kang's conjecture, the {base graph} $G$ is assumed to be triangle-free \ep{which, of course, implies that $H$ is triangle-free as well}. In principle, it is possible that $H$ is triangle-free while $G$ is not, and it seems that the conclusion of Conjecture~\ref{conj:CK} could hold even then. We suspect that Conjecture~\ref{conj:AKS} should also hold in the following stronger form:
    
    \begin{conj}
        For every graph $F$, there is a constant $c_F > 0$ such that the following holds. Let $G$ be a graph and let $\mathcal{H} = (L,H)$ be a DP-cover of $G$. If $H$ is $F$-free and has maximum degree $d \geq 2$ and if $|L(u)| \geq c_F d /\log d$ for all $u \in V(G)$, then $G$ admits a proper $\mathcal{H}$-coloring.
    \end{conj}

    After this discussion, we are now ready to state our main result:

    \begin{theo}\label{mainTheorem}
        There is a constant $\alpha > 0$ such that for every $\eps > 0$, there is $d_0 \in \N$ such that the following holds. Suppose that $d$, $s$, $t \in \N$ satisfy
        \[
            d \geq d_0, \quad s \leq d^{\alpha \eps}, \quad \text{and} \quad  t \leq \frac{\alpha\eps\log d}{\log \log d}.
        \]
        If $G$ is a graph and $\mathcal{H} = (L,H)$ is a DP-cover of $G$ such that:
        \begin{enumerate}[label=\ep{\normalfont\roman*}]
            \item $H$ is $K_{s,t}$-free,
            \item $\Delta(H) \leq d$, and
            \item $|L(u)| \geq (1+\eps)d/\log d$ for all $u \in V(G)$,
        \end{enumerate}
        then $G$ has a proper $\mathcal{H}$-coloring.
    \end{theo}
    
    If $F$ is an arbitrary bipartite graph with parts of size $s$ and $t$, then an $F$-free graph is also $K_{s,t}$-free. Thus, Theorem~\ref{mainTheorem} yields the following result for large enough $d_0$ as a function of $s$, $t$, and $\eps$: 
    \begin{corl}\label{mainCorollary}
        For every bipartite graph $F$ and $\eps > 0$, there is $d_0 \in \N$ such that the following holds. Let $d \geq d_0$. Suppose $\mathcal{H} = (L,H)$ is a DP-cover of $G$ such that:
        \begin{enumerate}[label=\ep{\normalfont\roman*}]
            \item $H$ is $F$-free,
            \item $\Delta(H) \leq d$, and
            \item $|L(u)| \geq (1+\eps)d/\log d$ for all $u \in V(G)$.
        \end{enumerate}
        Then $G$ has a proper $\mathcal{H}$-coloring.
    \end{corl}

    Setting $d = \Delta(G)$ in Corollary~\ref{mainCorollary} gives an extension of Theorem~\ref{theo:col} to DP-coloring:
    \begin{corl}\label{corl:DP}
        For every bipartite graph $F$ and $\eps > 0$, there is $\Delta_0 \in \N$ such that every $F$-free graph $G$ with maximum degree $\Delta \geq \Delta_0$ satisfies $\chi_{DP}(G) \leq (1 + \eps)\Delta/\log \Delta$. 
    \end{corl}
    
    We close this introduction with
    a few words about the proof of Theorem~\ref{mainTheorem}. To find a proper $\mathcal{H}$-coloring of $G$ we employ a variant of the so-called ``R\"odl Nibble'' method, in which we randomly color a small portion of $V(G)$ and then iteratively repeat the same procedure with the vertices that remain uncolored. \ep{See \cite{Nibble} for a recent survey on this method.} Throughout the iterations, both the maximum degree of the cover graph and the minimum list size are decreasing, but we show that the former is decreasing at a faster rate than the latter. Thus, we eventually arrive at a situation where $\Delta(H) \ll |L(v)|$ for all $v \in V(G)$, and then it is easy to complete the coloring. The specific procedure in our proof is essentially the same as the one used by Kim \cite{Kim95} \ep{see also \cite[Chapter 12]{MolloyReed}} to bound the chromatic number of graphs of girth at least $5$, suitably modified for the DP-coloring framework. We describe it in detail in \S\ref{sectionOutline}. The main novelty in our analysis is in the proof of Lemma~\ref{degreeConcentration}, which allows us to control the maximum degree of the cover graph after each iteration. This is the only part of the proof that relies on the assumption that $H$ is $K_{s,t}$-free. The proof of Lemma~\ref{degreeConcentration} involves several technical ingredients, which we explain in \S\ref{sectionProofOfConcentration}. In \S\ref{sectionIterations}, we put the iterative process together and verify that the coloring can be completed.

\section{Preliminaries}
    In this section we outline the main probabilistic tools that will be used in our arguments. We start with the symmetric version of the Lov\'asz Local Lemma. 

    \begin{theo}[{Lov\'asz Local Lemma; \cite[\S4]{MolloyReed}}]\label{LLL}
        Let $A_1$, $A_2$, \ldots, $A_n$ be events in a probability space. Suppose there exists $p \in [0, 1)$ such that for all $1 \leq i \leq n$ we have $\P[A_i] \leq p$. Further suppose that each $A_i$ is mutually independent from all but at most $d_{LLL}$ other events $A_j$, $j\neq i$ for some $d_{LLL} \in \N$. If $4pd_{LLL} \leq 1$, then with positive probability none of the events $A_1$, \ldots, $A_n$ occur.
    \end{theo}
    
    Aside from the Local Lemma, we will require several concentration of measure bounds. The first of these is the Chernoff Bound for binomial random variables. We state the two-tailed version below:

    \begin{theo}[{Chernoff; \cite[\S5]{MolloyReed}}]\label{chernoff}
        Let $X$ be a binomial random variable on $n$ trials with each trial having probability $p$ of success. Then for any $0 \leq \xi \leq \E[X]$, we have
        \begin{align*}
            \P\Big[\big|X - \E[X]\big| \geq \xi\Big] < 2\exp{\left(-\frac{\xi^2}{3\E[X]}\right)}. 
        \end{align*}
    \end{theo}

    We will also take advantage of two versions of Talagrand's inequality. The first version is the standard one: 
\begin{theo}[{Talagrand's Inequality; \cite[\S10.1]{MolloyReed}}]\label{Talagrand}
    Let $X$ be a non-negative random variable, not identically 0, which is a function of $n$ independent trials $T_1$, \ldots, $T_n$. Suppose that $X$ satisfies the following for some $\gamma$, $r > 0$: 
\begin{enumerate}[label=\ep{\normalfont{}T\arabic*}]
    \item Changing the outcome of any one trial $T_i$ can change $X$ by at most $\gamma$.
    \item For any $s>0$, if $X \geq s$ then there is a set of at most $rs$ trials that certify $X$ is at least $s$.
\end{enumerate}
Then for any $0 \leq \xi \leq \E[X]$, we have
\begin{align*}
    \P\Big[\big|X-\E[X]\big| \geq \xi + 60\gamma\sqrt{r\E[X]}\Big]\leq 4\exp{\left(-\frac{\xi^2}{8\gamma^2r\E[X]}\right)}.
\end{align*}
\end{theo}

The second version of Talagrand's inequality we will use was developed by Bruhn and Joos \cite{ExceptionalTal}. We refer to it as \emph{Exceptional Talagrand's Inequality}. In this version, we are allowed to discard a small ``exceptional'' set of outcomes before constructing certificates.

\begin{theo}[{Exceptional Talagrand's Inequality \cite[Theorem 12]{ExceptionalTal}}]\label{ExceptionalTalagrand}
    Let $X$ be a non-negative random variable, not identically 0, which is a function of $n$ independent trials $T_1$, \ldots, $T_n$, and let $\Omega$ be the set of outcomes for these trials. Let $\Omega^* \subseteq \Omega$ be a measurable subset, which we shall refer to as the \emphd{exceptional set}. Suppose that $X$ satisfies the following for some $\gamma>1$, $s>0$:
    \begin{enumerate}[label=\ep{\normalfont{}ET\arabic*}]
        \item\label{item:ET1} For all $q>0$ and every outcome $\omega \notin \Omega^*$, there is a set $I$ of at most $s$ trials such that $X(\omega') > X(\omega) - q$ whenever $\omega' \not\in \Omega^*$ differs from $\omega$ on fewer than $q/\gamma$ of the trials in $I$. 
        \item\label{item:ET2} $\P(\Omega^*) \leq M^{-2}$, where $M = \max\{\sup X, 1\}$.
    \end{enumerate}
    Then for every $\xi > 50\gamma\sqrt{s}$, we have:
    \begin{align*}
        \P\Big[\big|X - \E[X]\big| \geq \xi\Big] \leq 4\exp{\left(-\frac{\xi^2}{16\gamma^2s}\right)} + 4\P(\Omega^*).
    \end{align*}
\end{theo}

Finally, we shall use the K\H{o}v\'ari--Sós--Turán theorem for $K_{s,t}$-free graphs:

\begin{theo}[K\H{o}v\'ari--Sós--Turán \cite{KST}; see also \cite{KST2}]\label{KST}
    Let $G$ be a bipartite graph with a bipartition $V(G) = X \sqcup Y$, where $|X| = m$, $|Y| = n$, and $m \geq n$. Suppose that $G$ does not contain a complete bipartite subgraph with $s$ vertices in $X$ and $t$ vertices in $Y$. 
    Then
    $|E(G)| \leq s^{1/t} m^{1-1/t} n + tm$.
\end{theo}

\section{The Wasteful Coloring Procedure}\label{sectionOutline}
To prove Theorem \ref{mainTheorem}, we will start by showing we can produce a partial $\mathcal{H}$-coloring of our graph with desirable properties. Before we do so, we introduce some notation used in the next lemma. When $\phi$ is a partial $\mathcal{H}$-coloring of $G$, we define $L_\phi(v) \defeq \{c \in L(v) \,:\, N_H(c) \cap \im(\phi) = \0 \}$.
Given parameters $d$, $\ell$, $\eta$, $\beta > 0$, we define the following functions:
\begin{align*}
    \keep(d, \ell, \eta) &\defeq \left(1 - \frac{\eta}{\ell}\right)^{d}, \\
    \uncolor(d, \ell, \eta) &\defeq 1  - \eta \, \keep(d, \ell, \eta),\\
    \ell'(d, \ell, \eta, \beta) &\defeq \keep(d, \ell, \eta)\, \ell- \ell^{1 - \beta}, \\
    d'(d, \ell, \eta, \beta) &\defeq \keep(d, \ell, \eta)\, \uncolor(d, \ell, \eta)\, d + d^{1 - \beta}.
\end{align*}
The meaning of this notation will become clear when we describe the randomized coloring procedure we use to prove the following lemma.

\begin{Lemma}\label{iterationTheorem} 
    There are $\tilde{d} \in \N$, $\tilde{\alpha} > 0$ with the following property. Let $\eta >0$, $d$, $\ell$, $s$, $t \in \N$  satisfy:
    \begin{enumerate}[label=\ep{\normalfont\arabic*}]
        \item $d \geq \tilde{d}$,
        \item $\eta\, d < \ell < 8d$,
        \item $s \leq d^{1/4}$,
        \item $t  \leq \dfrac{\tilde{\alpha}\log d}{\log \log d}$, 
        \item $\dfrac{1}{\log^5d} < \eta < \dfrac{1}{\log d}.$
    \end{enumerate}Then whenever $G$ is a graph and $\mathcal{H} = (L, H)$ is a DP-cover of $G$ such that
    \begin{enumerate}[label=\ep{\normalfont\arabic*},resume]
        \item $H$ is $K_{s,t}$-free,
        \item $\Delta(H) \leq d$,
        \item $|L(v)| \geq \ell$ for all $v \in V(G)$,
    \end{enumerate}
    there exists a partial proper $\mathcal{H}$-coloring $\phi$ and 
    an assignment of subsets $L'(v) \subseteq L_\phi(v)$ to each $v \in V(G) \setminus \mathrm{dom}(\phi)$ such that, setting
    \[
        G' \defeq G\left[V(G)\setminus \mathrm{dom}(\phi)\right] \quad \text{and} \quad H' \defeq H\left[\bigcup_{v \in V(G')} L'(v)\right],
    \]
    we get that for all $v \in V(G')$, $c \in L'(v)$, and $\beta = 1/(25t)$:
    \begin{align*}
        |L'(v)| \,\geq\, \ell'(d, \ell, \eta, \beta) \quad \text{and} \quad \deg_{H'}(c) \,\leq\, d'(d, \ell, \eta, \beta).
    \end{align*}
\end{Lemma}


    


To prove Lemma \ref{iterationTheorem}, we will carry out a variant of the the ``Wasteful Coloring Procedure,'' as described in \cite[Chapter 12]{MolloyReed}. As mentioned in the introduction, essentially the same procedure was used by Kim \cite{Kim95} to bound the chromatic number of graphs of girth at least $5$. 
We describe this procedure in terms of DP-coloring below:\\

\begin{mdframed}
\textit{Wasteful Coloring Procedure}

\medskip

\noindent \textbf{Input:} A graph $G$ with a DP-cover $\mathcal{H} = (L,H)$ and a parameter $\eta \in [0,1]$.

\smallskip

\noindent \textbf{Output:} 
A proper partial $\mathcal{H}$-coloring $\phi$ and subsets $L'(v) \subseteq L_\phi(v)$ for all $v \in V(G) \setminus \dom(\phi)$. 

\medskip

\begin{enumerate}[wide]
    \item Generate a random subset $A \subseteq V(G)$ by putting each vertex $v \in V(G)$ into $A$ independently with probability $\eta$. The vertices in $A$ are said to be \emphdef{activated}. 
    \item Independently for each $v \in A$, pick a color $\col(v) \in L(v)$ uniformly at random. We say that $v$ is  \emphdef{assigned} the color $\col(v)$ and let $\col(A) \subseteq V(H)$ be the set of the assigned colors.
    \item For each vertex $v \in V(G)$, let
    \[
        K(v) \defeq \set{c \in L(v) \,:\, N_H(c) \cap \col(A) = \0}.
    \]
    We say that the colors in $K(v)$ are \emphd{kept} by $v$, and that $K \defeq \bigcup_{v \in V(G)}K(v)$ is the set of \emphdef{kept} colors. We also say that the colors in the set $V(H)\setminus K$ are \emphdef{removed}. 
    \item For each $v \in A$, if $\col(v) \in K(v)$, then we set $\phi(v) \defeq \col(v)$. For all other vertices, $\phi(v)$ is undefined, so we set $\phi(v) \defeq \mathsf{blank}$. We call such vertices \emphdef{uncolored}. 
    \item For each $v \in V(G) \setminus \dom(\phi)$, we let $L'(v) \defeq K(v)$.
\end{enumerate}
\end{mdframed}
\hspace{2pt}

In \S\S\ref{sectionProofOfIterationTheorem} and \ref{sectionProofOfConcentration}, we will show that, with positive probability, the output of the Wasteful Coloring Procedure satisfies the conclusion of Lemma~\ref{iterationTheorem}. With this procedure in mind, we can now provide an intuitive understanding for the functions defined in the beginning of this section. Suppose $G$, $\mathcal{H} = (L,H)$ satisfy $|L(v)| = \ell$ and $\Delta(H) = d$. If we run the Wasteful Coloring Procedure with these $G$ and $\mathcal{H}$, then $\keep(d,\ell, \eta)$ is the probability that a color $c \in L(v)$ is kept by $v$ \ep{i.e., $c \in K(v)$}, 
while $\uncolor(d, \ell, \eta)$ is approximately the probability that a vertex $v \in V(G)$ is uncolored \ep{i.e., $\phi(v) = \mathsf{blank}$}. The details of these calculations are given in \S\ref{sectionProofOfIterationTheorem}. 
Note that, assuming the terms $\ell^{1-\beta}$ and $d^{1-\beta}$ in the definitions of $\ell'(d, \ell, \eta, \beta)$ and $d'(d, \ell, \eta, \beta)$ are small, we can write
\[
    \frac{d'(d, \ell, \eta, \beta)}{\ell'(d, \ell, \eta, \beta)} \,\approx\, \uncolor(d, \ell, \eta) \frac{d}{\ell}.
\]
In other words, an application of Lemma~\ref{iterationTheorem} reduces the ratio $d/\ell$ roughly by a factor of $\uncolor(d, \ell, \eta)$. In \S\ref{sectionIterations} we will show that Lemma~\ref{iterationTheorem} can be applied iteratively to eventually make the ratio $d/\ell$ less than, say, $1/8$, after which the coloring can be completed using the following proposition:

\begin{prop}\label{finalBlow}
Let $G$ be a graph with a $DP$-cover $\mathcal{H} = (L,H)$ such that $|L(v)| \geq 8d$ for every $v \in V(G)$, where $d$ is the maximum degree of $H$. Then, there exists a proper $\mathcal{H}$-coloring of $G$.
\end{prop}

This proposition is standard and proved using the Lov\'asz Local Lemma. Its proof in the DP-coloring framework can be found, e.g., in \cite[Appendix]{JMTheorem}. 



For the reader already familiar with some of the applications the ``R\"odl Nibble'' method to graph coloring problems, let us make a comment about one technical feature of our Wasteful Coloring Procedure. For the analysis of constructions of this sort, it is often beneficial to assume that every color has the same probability of being kept. It is clear, however, that in our procedure the probability that a color $c \in V(H)$ is kept depends on the degree of $c$ in $H$: the larger the degree, the higher the chance that $c$ gets removed. The usual way of addressing this issue is by introducing additional randomness in the form of ``equalizing coin flips'' that artificially increase the probability of removing the colors of low degree. \ep{See, for example, the procedure in \cite[Chapter 12]{MolloyReed}.} However, it turns out that we can avoid the added technicality of dealing with equalizing coin flips by leveraging the generality of the DP-coloring framework. Namely, by replacing $H$ with a supergraph, we may always arrange $H$ to be $d$-regular \ep{see Proposition~\ref{regular}}. This allows us to assume that every color has the same probability of being kept, even without extra coin flips. This way of simplifying the analysis of probabilistic coloring constructions was introduced by Bonamy, Perrett, and Postle in \cite{GraphEmbedding} and nicely exemplifies the benefits of working with DP-colorings compared to the classical list-coloring setting.

\section{Proof of Lemma \ref{iterationTheorem}}\label{sectionProofOfIterationTheorem}

In this section, we present the proof of Lemma \ref{iterationTheorem}, apart from one technical lemma that will be established in \S\ref{sectionProofOfConcentration}. We start with the following proposition which allows us to assume that the given DP-cover of $G$ is $d$-regular. 

\begin{prop}\label{regular}
    Let $G$ be a graph and $(L,H)$ be a DP-cover of $G$ such that $\Delta(H) \leq d$ and $H$ is $K_{s,t}$-free for some $d$, $s$, $t \in \N$. Then there exist a graph $G^*$ and a DP-cover $(L^*, H^*)$ of $G^\ast$ such that the following statements hold:
    \begin{itemize}
        \item $G$ is a subgraph of $G^*$,
        \item $H$ is a subgraph of $H^*$,
        \item for all $v \in V(G)$, $L^\ast(v) = L(v)$,
        \item $H^*$ is $K_{s,t}$-free,
        \item $H^*$ is $d$-regular.
    \end{itemize}
\end{prop}

\begin{proof}
Set $N = \sum_{c \in V(H)}(d-\deg_H(c))$ and let $\Gamma$ be an $N$-regular graph with girth at least $5$. \ep{Such $\Gamma$ exists by \cite{GirthRegular1, GirthRegular2}.} 
Without loss of generality, we may assume that $V(\Gamma) = \set{1,\ldots,k}$, where $k \defeq |V(\Gamma)|$. Take $k$ vertex-disjoint copies of $G$, say $G_1$, \ldots, $G_k$, and let $(L_i, H_i)$ be a DP-cover of $G_i$ isomorphic to $(L,H)$. 
Define $X_i \defeq \{c\in V(H_i)\,:\, \deg_{H_i}(c) < d\}$ for every $1\leq i \leq k$. The graphs $G^\ast$ and $H^\ast$ are obtained from the disjoint unions of $G_1$, \ldots, $G_k$ and $H_1$, \ldots, $H_k$ respectively by performing the following sequence of operations once for each edge $ij \in E(\Gamma)$, one edge at a time:
\begin{enumerate}
    \item Pick arbitrary vertices $c \in X_i$ and $c' \in X_j$.
    \item Add the edge $cc'$ to $E(H^*)$ and the edge $L_i^{-1}(c)L_j^{-1}(c')$ to $E(G^*)$.
    \item If $\deg_{H^*}(c) = d$, remove $c$ from $X_i$.
    \item If $\deg_{H^*}(c') = d$, remove $c'$ from $X_j$.
\end{enumerate}
Since $\Gamma$ is $N$-regular, throughout this process the sum $\sum_{c \in V(H_i)} (d - \deg_{H^\ast}(c))$ decreases exactly $N$ times, which implies that the resulting graph $H^\ast$ is $d$-regular. 
Furthermore, since $\Gamma$ has girth at least $5$ and $H$ is $K_{s,t}$-free, $H^*$ is also $K_{s,t}$-free. 
Hence, if we define $L^*\colon V(G^*) \to 2^{V(H^*)}$ so that $L^*(v) = L_i(v)$ for all $v \in V(G_i)$, 
then $(L^*,H^*)$ is a DP-cover of $G^*$ satisfying all the requirements. 
\end{proof}


Suppose $d$, $\ell$, $s$, $t$, $\eta$, and a graph $G$ with a DP-cover $\mathcal{H} = (L, H)$ satisfy the conditions of Lemma~\ref{iterationTheorem}. By removing some vertices from $H$ if necessary, we may assume that $|L(v)| = \ell$ for all $v \in V(G)$. Furthermore, by Proposition~\ref{regular}, we may assume that $H$ is $d$-regular. Since we may delete all the edges of $G$ whose corresponding matchings in $H$ are empty, we may also assume that $\Delta(G) \leq \ell d$. Suppose we have carried out the Wasteful Coloring Procedure with these $G$ and $\mathcal{H}$. As in the statement of Lemma~\ref{iterationTheorem}, we let
\[
        G' \defeq G[V(G)\setminus \mathrm{dom}(\phi)] \quad \text{and} \quad H' \defeq H\left[\bigcup_{v \in V(G')} L'(v)\right].
    \]
For each $v \in V(G)$ and $c \in V(H)$, we define the random variables \[\ell'(v) \,\defeq\, |K(v)| \quad \text{and} \quad d'(c) \,\defeq\, |N_H(c)\cap V(H')|.\] Note that if $v \in V(G')$, then $\ell'(v) = |L'(v)|$; similarly, if $c \in V(H')$, then $d'(c)= \deg_{H'}(c)$. As in Lemma~\ref{iterationTheorem}, we let $\beta \defeq 1/(25t)$. For the ease of notation, we will write $\keep$ to mean $\keep(d,\ell,\eta)$, $\uncolor$ to mean $\uncolor(d, \ell, \eta)$, etc. 
We will show, for $d$ large enough, that:

\begin{Lemma}\label{listExpectation}
    For all $v \in V(G)$, $\E[\ell'(v)] = \keep\, \ell$,
\end{Lemma}

\begin{Lemma}\label{listSizeConcentration}
    For all $v \in V(G)$, 
    $\P\Big[\big|\ell'(v) - \E[\ell'(v)]\big| > \ell^{1-\beta}\Big] \leq d^{-100}$.
\end{Lemma}

\begin{Lemma}\label{degreeExpectation}
    For all $c \in V(H)$, $\E[d'(c)] \leq \keep\, \uncolor\, d + d/\ell$,
\end{Lemma}

\begin{Lemma}\label{degreeConcentration}
    For all $c \in V(H)$, $\P\big[d'(c) > \E[d'(c)] - d/\ell + d^{1-\beta}\big] \leq d^{-100}.$
\end{Lemma}

Together, these lemmas will allow us to complete the proof of Lemma \ref{iterationTheorem}, as follows.

\begin{proof}[Proof of Lemma \ref{iterationTheorem}]
Take $\tilde d$ so large that Lemmas \ref{listExpectation}--\ref{degreeConcentration} hold. 
Define the following random events for every vertex $v \in V(G)$ and every color $c\in V(H)$: 
    \[A_v \defeq \left[\ell'(v) \leq \ell'\right] \quad \text{and} \quad B_{c} \defeq \left[d'(c) \geq d'\right].\]
We will use the \LLL, Theorem \ref{LLL}. By Lemma \ref{listExpectation} and Lemma \ref{listSizeConcentration}, we have:
\begin{align*}
    \P[A_v] &= \P\big[\ell'(v) \leq \keep \, \ell - \ell^{1-\beta}\big]\\
    &= \P\big[\ell'(v) \leq \E[\ell'(v)] - \ell^{1-\beta}\big] \\
    & \leq d^{-100}.
\end{align*}
By Lemma \ref{degreeExpectation} and Lemma \ref{degreeConcentration}, we have:
\begin{align*}
    \P[B_{c}] &= \P\big[d'(c) \geq \keep \, \uncolor \, d + d^{1-\beta}\big]\\
    & \leq \P\big[d'(c) \geq \E[d'(c)] - (d/\ell)+  d^{1-\beta}\big] \\
    & \leq d^{-100}.
\end{align*}
Let $p \defeq d^{-100}$.
Note that events $A_v$ and $B_{c}$ are mutually independent from events of the form $A_u$ and $B_{c'}$ where $c' \in L(u)$ and $u \in V(G)$ is at distance at least $5$ from $v$. Since we are assuming that $\Delta(G) \leq \ell d$, there are at most $1 + (\ell d)^4$ vertices in $G$ of distance at most $4$ from $v$. For each such vertex $u$, there are $\ell + 1$ events corresponding to $u$ and the colors in $L(u)$, so we can let $d_{LLL} \defeq (\ell+1)(1 + (\ell d)^4) = O(d^9)$. Assuming $d$ is large enough, we have
\[
    4pd_{LLL} \leq 1,
\]
so, by Theorem \ref{LLL}, with positive probability none of the events $A_v$, $B_{c}$ occur, as desired. 
\end{proof}

The proofs of Lemmas \ref{listExpectation}--\ref{degreeExpectation} are fairly straightforward and similar to the corresponding parts of the argument in the girth-$5$ case \ep{see \cite[Chapter 12]{MolloyReed}}. We present them here.

\begin{proof}[Proof of Lemma \ref{listExpectation}]
    Consider any $c \in L(v)$. We have $c \in K(v)$ exactly when $N_H(c) \cap \col(A) = \0$, i.e., when no neighbor of $c$ is assigned to its underlying vertex. The probability of this event is $\left(1 -\eta/\ell\right)^{d} = \keep$. By the linearity of expectation, it follows that $\E[\ell'(v)] = \keep\, \ell$.
\end{proof}

\begin{proof}[Proof of Lemma \ref{listSizeConcentration}]
    It is easier to consider the random variable $r(v) \defeq \ell - \ell'(v)$, the number of colors removed from $L(v)$. 
    We will use Theorem \ref{Talagrand}, Talagrand's Inequality. Order the colors in $L(u)$ for each $u \in N_G(v)$ arbitrarily. Let $T_u$ be the random variable that is equal to $0$ if $u \not \in A$ and $i$ if $u \in A$ and $\col(u)$ is the $i$-th color in $L(u)$. Then $T_u$, $u \in N_G(v)$ is a list of independent trials whose outcomes determine $r(v)$. Changing the outcome of any one of these trials can affect $r(v)$ at most by $1$. 
    Furthermore, if $r(v) \geq s$ for some $s$, then this fact can be certified by the outcomes of $s$ of these trials. Namely, for each removed color $c \in L(v) \setminus K(v)$, we take the trial $T_u$ corresponding to any vertex $u \in N_G(v)$ such that $u \in A$ and $\col(u)$ is adjacent to $c$ in $H$. 
    Thus, we can now apply Theorem \ref{Talagrand} with $\gamma=1$, $r=1$ to get:
\begin{align*}
    \P\Big[\big|\ell'(v) - \E[\ell'(v)]\big| > \ell^{1-\beta}\Big] &=\P\Big[\big|r(v) - \E[r(v)]\big| > \ell^{1-\beta}\Big] \\
    &= \P\Big[\big|r(v) - \E[r(v)]\big| > \frac{\ell^{1-\beta}}{2} + \frac{\ell^{1-\beta}}{2}\Big] \\
    &\leq \P\Big[\big|r(v) - \E[r(v)]\big| > \frac{\ell^{1-\beta}}{2} + 60\sqrt{\E[r(v)]}\Big] \\
    &\leq 4\exp\left(-\frac{\ell^{2(1-\beta)}}{32\, (1-\keep)\, \ell}\right) \\
    &\leq 4\exp{\left(-\frac{\ell^{1-2\beta}}{32}\right)} \\
    &\leq 4\exp{\left(-\frac{(d/\log^5d)^{1-2\beta}}{32}\right)} \\
    &\leq d^{-100},
\end{align*}
where the first and last inequalities hold for $d$ large enough.
\end{proof}

\begin{proof}[Proof of Lemma \ref{degreeExpectation}]

Let $u \in N_G(v)$ and $c'\in L(u) \cap N_H(c)$. We need to bound the probability that $\phi(u) = \blank$ and $c' \in K(u)$. We split into the following cases.

\smallskip

\noindent \textbf{Case 1:} $u \notin A$ and $c' \in K(u)$. This occurs with probability $(1-\eta) \keep$.

\smallskip

\noindent \textbf{Case 2:} $u \in A$, $\col(u)=c'' \neq c'$, $\phi(u) = \blank$, and $c' \in K(u)$. In this case, there must be some $w \in N_G(u)$ such that $\col(w) \sim c''$. Since $c' \in K(u)$, we must have $\col(w) \not\sim c'$. For each $w \in N_G(u)$, 
\begin{align*}
    \P\big[\col(w) \sim c'' \,|\, \col(w) \not\sim c'\big] 
    \,=\, \left(\frac{\eta}{\ell}\right)/\left(1 - \frac{\eta}{\ell}\right)
    \,=\, \frac{\eta}{\ell - \eta}.
\end{align*}
Therefore, we can write 
\begin{align*}
    &\P\big[\phi(u) = \blank\, |\, \col(u) = c'',\, c' \in K(u)\big]\\
    &=\, 1 - \left(1 - \frac{\eta}{\ell - \eta}\right)^d\\
    &=\, 1 - \keep \, \left(1 - \frac{\eta^2}{(\ell - \eta)^2}\right)^d\\
    &\leq\, 1 - \keep + \keep \frac{d\eta^2}{(\ell-\eta)^2}\\
    &\leq\, 1 - \keep + \frac{1}{\ell},
\end{align*}
where the last inequality follows since $\keep \leq 1$, $\eta d < \ell$, $\eta < 1/\log d$, and $d$ is large enough.

Putting the two cases together, we have:
\begin{align*}
    &\P\big[\phi(u) = \blank, \, c' \in K(u)\big]
    \,\\&\leq\, (1-\eta)\, \keep + \eta\, \left(1 - \frac{1}{\ell}\right)\, \keep \, \left(1 - \keep + \frac{1}{\ell}\right) \\
    &\leq\, \keep\, \uncolor + \frac{1}{\ell}.
\end{align*}
Finally, by linearity of expectation, we conclude that
\begin{align*}
    \E[d'(c)] \,\leq\, d\left(\keep \, \uncolor + \frac{1}{\ell}\right) \,=\, \keep\, \uncolor\, d + \frac{d}{\ell},
\end{align*}
proving the lemma.
\end{proof}

    The proof of Lemma \ref{degreeConcentration} is quite technical and will be given in \S\ref{sectionProofOfConcentration}. It is the only part of our argument that relies on the fact that $H$ is $K_{s,t}$-free. To explain why proving Lemma~\ref{degreeConcentration} is difficult, consider an arbitrary color $c \in V(H)$. The value $d'(c)$ depends on which of the neighbors of $c$ in $H$ are kept. This, in turn, is determined by what happens to the neighbors of the neighbors of $c$. Since we are only assuming that $H$ is $K_{s,t}$-free, the neighborhoods of the neighbors of $c$ can overlap with each other. 
    Roughly speaking, we will need to carefully analyze the structure of these overlaps to make sure that Talagrand's inequality can be applied.  


\section{Proof of Lemma \ref{degreeConcentration}}\label{sectionProofOfConcentration}

Throughout this section, we shall use the following parameters, where $t$ is given in the statement of Lemma \ref{iterationTheorem}:
\[
    \beta \defeq \frac{1}{25t}, \quad \beta_1 \defeq \frac{1}{20t}, \quad  \beta_2 \defeq \frac{1}{15t}, \quad \delta \defeq \frac{1}{3t}, \quad \delta_2 \defeq \frac{1}{10t}, \quad \tau \defeq \frac{4}{9t}.
\]
Fix a vertex $v \in V(G)$ and a color $c \in L(v)$. We need too show that, with high probability, the random variable $d'(c)$ does not significantly exceed its expectation. To this end, we make the following definitions:
\begin{align*}
    \mathcal{K} & \defeq \{c' \in N_H(c) : N_H(c')\cap \col(A) = \0\}, \\
    \mathcal{U}  & \defeq \{c' \in N_H(c) : \phi(L^{-1}(c')) = \blank\}. 
\end{align*}
Then $d'(c) = |\mathcal{U} \cap \mathcal{K}|$. 
We will show that $|\mathcal{U}|$ is highly concentrated and prove that, with high probability, $|\mathcal{U} \setminus \mathcal{K}|$ is not much lower than its expected value. Using the identity $|\mathcal{U} \cap \mathcal{K}| = |\mathcal{U}| - |\mathcal{U}\setminus \mathcal{K}|$ will then give us the desired upper bound on $d'(c)$.

\begin{Lemma}\label{uncolorConcentration}
    $\P\bigg[\Big||\mathcal{U}|- \E\big[|\mathcal{U}|\big]\Big| \geq d^{1-\beta_1}\bigg] \leq d^{-110}$.
\end{Lemma}

\begin{proof}
We use Theorem \ref{ExceptionalTalagrand}, Exceptional Talagrand's Inequality. 
Let $V_c \defeq L^{-1}(N_H(c))$. In other words, $V_c$ is the set of neighbors of $v$ whose lists include a color corresponding to $c$. Then the set $\mathcal{U}$ is determined by the coloring outcomes of the vertices in $S \defeq V_c\cup N_G(V_c)$. More precisely, as in the proof of Lemma~\ref{listSizeConcentration}, we arbitrarily order the colors in $L(u)$ for each $u \in S$ and let $T_u$ be the random variable that is equal to $0$ if $u \not \in A$ and $i$ if $u \in A$ and $\col(u)$ is the $i$-th color in $L(u)$. Then $T_u$, $u \in S$ is a list of independent trials whose outcomes determine $|\mathcal{U}|$. Let $\Omega$ be the set of outcomes of these trials. 
Let $C \defeq 25$ 
and define $\Omega^* \subseteq \Omega$ to be the set of all outcomes in which there is a color $c' \in L(S)$ such that $|N_H(c') \cap L(V_c) \cap \col(A) | \geq C{\log d}$. 
We claim that $|\mathcal{U}|$ satisfies conditions \ref{item:ET1} and \ref{item:ET2} of Theorem \ref{ExceptionalTalagrand} with $s = 2d$ and $\gamma = 1 + C \log d$. 

To verify \ref{item:ET1}, take $q > 0$ and outcome $\omega \notin \Omega^*$. 
Each vertex $u \in L^{-1}(\mathcal{U})$ satisfies $u \notin A$ or there exists $w\in N_G(u)$ such that $u$, $w \in A$ and $\col(w) \sim \col(u)$. We call such $w$ a \emphd{conflicting neighbor} of $u$. Form a subset $I$ of trials by including, for each $u \in L^{-1}(\mathcal{U})$, 
the trial $T_u$ itself and
, if applicable, the trial $T_w$ corresponding to any one conflicting neighbor $w$ of $u$.  
Since $|\mathcal{U}| \leq d$, we have $|I| \leq 2d = s$. 
Now suppose that $\omega' \not\in \Omega^\ast$ satisfies $|\mathcal{U}(\omega')| \leq |\mathcal{U}(\omega)| - q$. For each vertex $u \in \mathcal{U}(\omega)\setminus \mathcal{U}(\omega')$, the outcomes of either the trial $T_u$ or the trial $T_w \in I$ of a conflicting neighbor $w$ of $u$ must be different in $\omega$ and in $\omega'$. Since $\omega \notin \Omega^\ast$, every $w \in S$ can be a conflicting neighbor of at most  
$C{\log d}$ vertices $u$. Therefore, $\omega'$ and $\omega$ must differ on at least 
$q/(1 + C{\log d})$ trials, as desired. 


It remains to show $\P\left[\Omega^*\right] \leq M^{-2}$, where $M = \max\{\sup |\mathcal{U}|, 1\}$. For any $c' \in L(S)$, the number of colors in $N_H(c') \cap L(V_c) \cap \col(A)$ 
is a binomial random variable with at most $d$ trials, each having probability $\eta/\ell$. Let $X_{c'}$ denote this random variable. Note that $\E[X_{c'}] = \eta d/\ell < 1$. By the union bound, we have
\begin{align*}
    \P\left[X_{c'} \geq C{\log d}\right] &\leq {d \choose \lceil C \log d \rceil} \left(\frac{\eta}{\ell}\right)^{\lceil C \log d \rceil} \leq \left(\frac{ed}{\lceil C \log d \rceil}\right)^{\lceil C \log d \rceil}  \left(\frac{\eta}{\ell}\right)^{\lceil C \log d \rceil} \\
    &\leq \left(\frac{e}{\lceil C \log d \rceil}\right)^{\lceil C \log d \rceil} \leq d^{-150},
\end{align*}
where the last inequality holds for $d$ large enough.
By the union bound and the fact that $\ell \leq 8d$ (by the assumptions of Lemma \ref{iterationTheorem}), we get 
\begin{align*}
    \P\left[\exists~ c' \in L(S) \text{ such that } X_{c'} \geq C \log d\right] \leq \ell |S| d^{-150} \leq d^{-125},
\end{align*}
where we use that $|S| \leq d + \ell d^2$ and $d$ is large enough. 
Thus $\P\left[\Omega^*\right] \leq d^{-125}$. Note that $M = \max\{\sup |\mathcal{U}|, 1 \} = \max\{ d, 1\} = d$, so $\P\left[\Omega^*\right] \leq 1/M^2$, for $d$ large enough. 
 
We can now use Exceptional Talagrand's Inequality. Let $\xi \defeq d^{1-\beta_1}$. Note that $\xi > 50\gamma\sqrt{s}$ for $d$ large enough. We can therefore write 
\begin{align*}
\P\bigg[\Big||\mathcal{U}|- \E\big[|\mathcal{U}|\big]\Big| \geq d^{1-\beta_1}\bigg] &\leq 4\exp{\left(-\frac{d^{2-2\beta_1}}{32(1+C\log d)^2d}\right)} + 4\P[\Omega^*]\\
&\leq 4\exp\left( -O\left(\frac{d^{1 - 2\beta_1}}{\log^2 d}\right)\right) + 4d^{-125} \\
&\leq d^{-110},
\end{align*}
for $d$ large enough.
\end{proof}

\begin{Lemma}\label{UminusKConcentration}
    $\P\Big[|\mathcal{U}\setminus \mathcal{K}| \geq \E\big[|\mathcal{U}\setminus \mathcal{K}|\big] - d^{1-\beta_1}\Big] \geq 1-d^{-110}.$
\end{Lemma}

It is here that we take advantage of the fact that $H$ is $K_{s,t}$-free. Before we proceed, we make a few definitions. Recall that $\delta = 1/(3t)$, where $t$ is given in the statement of Lemma \ref{iterationTheorem}. Let $N^2_H(c)$ denote the set of all colors $c'' \in V(H)$ that are joined to $c$ by at least one path of length exactly $2$ \ep{there may also be other paths joining $c$ and $c''$}. We say:
\begin{align*}
    \text{$c'' \in N_H^2(c)$ is} &\text{ \emphdef{bad} if $c''$ has at least $d^{1-\delta}$ neighbors in $N_H(c)$},\\ &\text{\emphdef{good} otherwise};\\
    \text{$c' \in N_H(c)$ is}
    & \text{ \emphdef{sad} if $c'$ has at least $d^{1-\delta}$ bad neighbors},\\ &\text{ \emphdef{happy} otherwise.}
\end{align*}
Let $\Bad$, $\Good$, $\Sad$, and $\Happy$ be the sets of bad, good, sad, and happy colors respectively. Note that, as we are not assuming that $H$ is triangle-free, it is possible that $N_H(c) \cap N_H^2(c) \neq \0$; in particular, a color can be both bad and sad. 
Bad colors are problematic from the point of view of Talagrand's inequality, as each of them can be responsible for the removal of a large number of colors from $\mathcal{K}$. Thankfully, we can use the \hyperref[KST]{K\H{o}v\'ari--Sós--Turán theorem} to argue that there are few sad colors, i.e., most colors in $N_H(c)$ have only a few bad neighbors.
\begin{claim}\label{sadBound}
The number of sad colors is at most $d^{1-\beta_2}$, where $\beta_2 = 1/(15t)$.
\end{claim}
\begin{proof}
Since every bad color has at least $d^{1-\delta}$ neighbors in $N_H(c)$, we have
\begin{align*}
    |\Bad| \leq \frac{2|E\big(N_H(c),N_H^2(c)\big)|}{d^{1-\delta}} \leq \frac{2d^2}{d^{1-\delta}} = 2d^{1+\delta}.
\end{align*}
Let $B$ be the bipartite graph with parts $X$ and $Y$, where $X = \Bad$ and $Y$ is a copy of $N_H(c)$ disjoint from $X$, with the edges in $B$ corresponding to those in $H$. 
Note that if a color $c' \in N_H(c) \cap N_H^2(c)$ is bad, then $B$ contains two copies of $c'$, one in $X$ and the other in $Y$. However, these two copies cannot be adjacent to each other, and hence every subgraph of $B$ isomorphic to $K_{s,t}$ must use only one copy of each color. Since $H$ is $K_{s,t}$-free, we conclude that $B$ is $K_{s,t}$-free as well. If we set $\hat{\eps} = 1/t$, $m = |\Bad|$, $n = d$, then, by the \hyperref[KST]{K\H{o}v\'ari--Sós--Turán theorem}, 
\begin{align*}
    |E(B)| \leq s^{\hat{\eps}}(2d^{1+\delta})^{1-\hat{\eps}}\, d + 2\,t\,d^{1+\delta} \leq 4d^{2 -3\hat{\eps}/4 + \delta - \delta\hat{\eps}}.
\end{align*}
On the other hand, since every sad color has at least $d^{1-\delta}$ bad neighbors, we see that
$$|\Sad|\, d^{1-\delta} \leq |E(B)| \leq 4 d^{2 - 3\hat{\eps}/4 + \delta - \delta\hat{\eps}}.$$
This implies that for $d$ large enough,
\begin{align*}
    |\Sad| & \leq 4d^{1 - 3\hat{\eps}/4 + 2\delta - \delta\hat{\eps}} \leq d^{1-\beta_2},
\end{align*}
as $3\hat\eps/4 - 2\delta +  \delta\hat\eps > 1/(12t) > \beta_2 > 0$.
\end{proof}

Instead of proving the desired one-sided concentration inequality for $|\mathcal{U}\setminus \mathcal{K}|$ directly, we will focus on a slightly different parameter. Let
\begin{align*}
    \widetilde{\mathcal{K}} \defeq \big\{c'\in \Happy : N_H(c') \cap \Good \cap \col(A) = \emptyset\big\}.
\end{align*}
In other words, $\widetilde{\mathcal{K}}$ is the set of all happy colors that do not have a good neighbor in $\col(A)$. Then 
\begin{align*}
    \mathcal{U} \setminus \widetilde{\mathcal{K}} =\{c' \in N_H(c) :~ & \phi(L^{-1}(c')) = \blank, \text{ and either} \\
    & \text{(1) $c' \in \Sad$, or }\\ 
    & \text{(2) $N_H(c') \cap \Good \cap \col(A) \neq \0$}\}.
\end{align*}

\begin{claim}\label{ZBound} Let $\widetilde Z \defeq |(\mathcal{U} \setminus \widetilde{\mathcal{K}}) \cap \Happy|$. Then
$\P\big[\widetilde Z \geq \E[\widetilde Z] - d^{1-\delta_2}\big] \geq 1-d^{-110}$, where $\delta_2 = 1/(10t)$.
\end{claim}
\begin{proof}\stepcounter{ForClaims} \renewcommand{\theForClaims}{\ref{ZBound}}
Recall that $\tau = 4/(9t)$. 
By definition, a good color can be responsible for the removal of at most $d^{1-\delta}$ colors from $\widetilde{\mathcal{K}}$. Unfortunately, this bound is still too large to apply Talagrand's inequality directly. Instead, we will first partition $\Happy$ into $k \defeq \lceil d^{1-\tau} \rceil$ sets $\Happy_1$, \ldots, $\Happy_k$ satisfying certain properties and then argue that the random variable $|(\mathcal{U} \setminus \widetilde{\mathcal{K}}) \cap \Happy_i|$ is highly concentrated for each $i$. The following subclaim states these properties and proves the existence of the partition.

\begin{subclaim}\label{partition}
There exists a partition of $\Happy$ into sets $\Happy_1$, \ldots, $\Happy_k$ such that the following hold for all $1\leq i \leq k$ and every $c''\in \Good$:
\begin{itemize}
    \item $\dfrac{d^\tau}{4} \leq |\Happy_i| \leq \dfrac{3d^{\tau}}{2}$,
    
    \item $|N_H(c'') \cap \Happy_i| \leq \dfrac{3d^{\tau - \delta}}{2}$.
\end{itemize}
\end{subclaim}
\begin{claimproof}[Proof of Subclaim \ref{partition}]
Independently for each $c' \in \Happy$, assign $c'$ to $\Happy_i$ uniformly at random. Let $s_i \defeq |\Happy_i|$. Then $s_i$ is a binomial random variable with at most $d$ and at least $d - d^{1-\beta_2}$ trials, each succeeding with probability $1/k$. 
We have $\E[s_i] = d/k \in [\frac{3d^\tau}{4}, d^\tau]$ as $3d^\tau/4 < (d - d^{1-\beta_2})/k \approx d^\tau - d^{\tau - \beta_2}$ for $d$ large enough, since $\tau > \beta_2$. By the Chernoff bound \ep{Theorem \ref{chernoff}}, we have:
$$\P\Big[\big|s_i - \E[s_i]\big| \geq \frac{d^\tau}{2}\Big] \leq 2\exp{\left(-\frac{d^{\tau}}{12}\right)}.$$
By the union bound and since $t \leq \tilde \alpha\dfrac{\log d}{\log \log d}$, we have the following for $\tilde \alpha$ small enough:
\begin{align}\label{p1}
    \P\left[\exists\, i:\, \big|s_i - \E[s_i]\big| \geq \frac{d^\tau}{2}\right] \leq 2\,k\,\exp{\left(-\frac{d^{\tau}}{12}\right)} \leq d^{-1}.
\end{align}
Now, for $c'' \in \Good$, let $r_i(c'')$ be the number of neighbors $c''$ has in $\Happy_i$. Then $r_i(c'')$ is a binomial random variable with at most $d^{1-\delta}$ trials (since $c''$ is good), each succeeding with probability $1/k$. Let $\Theta$ be a binomial random variable with exactly $\lfloor d^{1-\delta} \rfloor$ trials, each succeeding with probability $1/k$. Note that $\E[r_i(c'')] \leq \E[\Theta] \leq d^{\tau -\delta}$ and $\E[\Theta] > d^{\tau- \delta}/2$. Then, by 
Theorem \ref{chernoff}, 
\begin{align*}
    \P\left[r_i(c'') \geq \frac{3d^{\tau - \delta}}{2}\right] \leq \P\left[\Theta \geq \E[\Theta] + \frac{d^{\tau - \delta}}{2}\right] \leq 2\exp{\left(-\frac{(d^{\tau - \delta}/2)^2}{3d^{\tau-\delta}}\right)} = 2\exp{\left(-\frac{d^{\tau-\delta}}{12}\right)}.
\end{align*}
By the union bound and since $t\leq \tilde \alpha\dfrac{\log d}{\log \log d}$, we have the following for $\tilde \alpha$ small enough
\begin{align}\label{p2}
    \P\left[\exists\, i,\, c''\in \Good :\, r_i(c'') \geq \frac{3d^{\tau - \delta}}{2}\right] \leq k\,d^2\,\exp{\left(-\frac{d^{\tau-\delta}}{12}\right)} \leq d^{-1}.
\end{align}
Putting together \eqref{p1} and \eqref{p2}, we obtain
$$\P\left[\Happy_1, \, \ldots, \,\Happy_k \text{ satisfy the conditions stated}\right] \geq 1 - 2d^{-1} > 0.$$
So, such a partition must exist.
\end{claimproof}

From here on out, we fix a partition $\Happy_1$, \ldots, $\Happy_k$ of $\Happy$ that satisfies the conclusions of Subclaim~\ref{partition}. For each $1 \leq i \leq k$, let $\widetilde{Z}_i\defeq |(\mathcal{U} \setminus \widetilde{\mathcal{K}})\cap \Happy_i|$. We will now use Exceptional Talagrand's Inequality (Theorem \ref{ExceptionalTalagrand}) to show that each random variable $\widetilde{Z}_i$ is highly concentrated.

\begin{subclaim}\label{subclaim:zj}
    For each $1 \leq i \leq k$, we have $\P\bigg[\left|\widetilde{Z}_i - \E[\widetilde{Z}_i]\right| \geq d^{\tau-\delta_2}\bigg] \leq d^{-120}$.
\end{subclaim}
\begin{claimproof}[Proof of Subclaim~\ref{subclaim:zj}]
For brevity, set $X \defeq \widetilde{Z}_i$. Let $D \defeq L^{-1}(\Happy_i)$ be the set of the underlying vertices of the colors in $\Happy_i$. 
The random variable $X$ is determined by the coloring outcomes of the vertices in $S \defeq D \cup N_G(D)$. More precisely, as in the proofs of Lemmas~\ref{listSizeConcentration} and \ref{uncolorConcentration}, we arbitrarily order the colors in $L(u)$ for each $u \in S$ and let $T_u$ be the random variable that is equal to $0$ if $u \not \in A$ and $i$ if $u \in A$ and $\col(u)$ is the $i$-th color in $L(u)$. Then $T_u$, $u \in S$ is a list of independent trials that determines $X$. 
Let $\Omega$ be the set of outcomes of these trials. Let $C \defeq 25$ and define $\Omega^* \subseteq \Omega$ to be the set of all outcomes in which there is a color $c'' \in L(S)$ such that $|N_H(c'')\cap \Happy_i \cap \col(A)| \geq C{\log d}$.
We claim that $X$ satisfies conditions \ref{item:ET1} and \ref{item:ET2} in Theorem~\ref{ExceptionalTalagrand} with $s = 9d^\tau/2$ and $\gamma = 1 + 3d^{\tau-\delta}/2 + C{\log d}$.

To verify \ref{item:ET1}, take $q > 0$ and $\omega \notin \Omega^*$. We form a set $I$ of at most $s$ trials as follows. Consider any color $c' \in \Happy_i$ that contributes towards $X$ and let $u \defeq L^{-1}(c')$. By definition, $\phi(u) = \blank$ and there is a good neighbor $c''$ of $c'$ with $c'' \in \col(A)$. Pick any such $c''$ and let $w \defeq L^{-1}(c'')$. We say that $w$ is the \emphd{conflicting neighbor of $u$ of Type I}. Next, since $\phi(u) = \blank$, we either have $u \notin A$, or there is $y \in N_G(u)$ such that $u$, $y \in A$ and $\col(y) \sim \col(u)$. Pick any such $y$ \ep{if it exists} and call it the \emphd{conflicting neighbor of $u$ of Type II}. Add the following trials to $I$:
\[
    T_u, \quad T_w, \quad T_y \text{ \ep{if applicable}}.
\]
Since $|\Happy_i| \leq 3d^\tau/2$, we have $|I| \leq 3|\Happy_i| \leq 9d^\tau/2 = s$.


Note that, for every vertex $w \in S$, there can be at most $3d^{\tau-\delta}/2$ vertices $u \in D$ such that $w$ is the conflicting neighbor of $u$ of Type I. Indeed, for $w$ to be the Type I conflicting neighbor of any vertex, it must be true that $w \in A$ and $\col(w) \in \Good$. Then, by Subclaim~\ref{partition}, $\col(w)$ has at most $3d^{\tau-\delta}/2$ neighbors in $\Happy_i$, as desired. Similarly, since $\omega \not \in \Omega^\ast$, for each vertex $y \in S$, there are at most $C {\log d}$ vertices $u \in D$ such that $y$ is the conflicting neighbor of $u$ of Type II.

Now suppose that $\omega' \not \in \Omega^\ast$ satisfies $X(\omega') \leq X(\omega) - q$. Consider any color $c' \in \Happy_i$ that contributes towards $X(\omega)$ but not $X(\omega')$ and let $u \defeq L^{-1}(c')$. Then either $T_u$ or at least one of the trials corresponding to the conflicting neighbors of $u$ must have different outcomes in $\omega$ and $\omega'$. The observations in the previous paragraph imply that the total number of trials on which $\omega$ and $\omega'$ differ must be at least $q/(1 + 3d^{\tau-\delta}/2 + C {\log d})$. 



It remains to show that $\P\left[\Omega^*\right] \leq M^{-2}$, where $M = \max\{\sup X, 1\}$. As in the proof of Lemma~\ref{uncolorConcentration}, we get
%
$\P\left[\Omega^*\right] \leq d^{-125}$ for $d$ large enough. Since $M = \max\{\sup X, 1 \} = \max\{ 3d^\tau/2, 1\} \leq d$, we conclude that $\P\left[\Omega^*\right] \leq 1/M^2$, as desired. 
 
We can now use Exceptional Talagrand's Inequality. Let $\xi = d^{\tau - \delta_2}$. Note that $2\delta - 2\delta_2 -\tau > 0 $ and $\xi > 50\gamma\sqrt{s}$ for $d$ large enough. We can therefore write 

\begin{align*}
\P\Big[\big|X - \E[X]\big| \geq d^{\tau - \delta_2}\Big] &\leq 4\exp{\left(-\frac{d^{2\tau-2\delta_2}}{16
\left(1 + 3d^{\tau-\delta}/2 + C {\log d}\right)^2
\left(\frac{9d^\tau}{2}\right)}\right)} + 4\P[\Omega^*]\\
&\leq 4\exp\left( - O\left(\frac{d^{2\tau-2\delta_2}}{d^{3\tau-2\delta}}\right)\right) + 4d^{-125} \\
&\leq 4\exp\left(-O\left(d^{2\delta - 2\delta_2 - \tau}\right)\right) + 4d^{-125}\\
&\leq d^{-120},
\end{align*}
for $d$ large enough and $\tilde\alpha$ small enough.
\end{claimproof}

Using Subclaim~\ref{subclaim:zj} and the union bound, we obtain
\begin{align*}
    \P\Big[\exists\,i:\, \widetilde{Z}_i \leq \E[\widetilde{Z}_i] - d^{\tau - \delta_2} \Big] \leq d^{1-\tau}d^{-120}\leq d^{-115}.
\end{align*}
Since $\widetilde{Z} \defeq \sum_{i = 1}^k \widetilde{Z}_i$, we conclude that
\begin{align*}
    \P\left[\widetilde{Z} \geq \E[\widetilde{Z}] - d^{1 - \delta_2}\right]
    \geq\P\left[\forall\,i:\, \widetilde{Z}_i \geq \E[\widetilde{Z}_i] - d^{\tau - \delta_2}\right]
    \geq 1-d^{-110},
\end{align*}
for $d$ large enough, as desired.
\end{proof}

Since $|(\mathcal{U} \setminus \widetilde{\mathcal{K}}) \cap \Happy| = \widetilde Z$, the value $\widetilde Z$ can differ from $|\mathcal{U}\setminus \widetilde{\mathcal{K}}|$ at most by the number of sad colors. Thus, by Claim~\ref{sadBound}, we have $0 \leq |(\mathcal{U} \setminus \widetilde{\mathcal{K}})|-\widetilde Z \leq d^{1-\beta_2}$, 
from which it follows that
\begin{align}\label{tildeZ}
    \E[\widetilde{Z}] &\geq \E\big[|\mathcal{U} \setminus \widetilde{\mathcal{K}}|\big] - d^{1-\beta_2}.
\end{align}
We now show that $\E\big[|\mathcal{U}\setminus \mathcal{K}|\big]$ is not much larger than $\E\big[|\mathcal{U} \setminus \widetilde{\mathcal{K}}|\big]$.

\begin{claim}\label{expectedValuesAreClose}
    $\E\big[|\mathcal{U}\setminus \mathcal{K}|\big] - \E\big[|\mathcal{U}\setminus \widetilde{\mathcal{K}}|\big] \leq d^{1-\delta}$.
\end{claim}
\begin{proof}
First note that
\begin{align*}
    |\mathcal{U}\setminus 
\mathcal{K}| - |\mathcal{U}\setminus \widetilde{\mathcal{K}}| \leq 
\big|(\mathcal{U}\setminus \mathcal{K})\cap \mathcal{\widetilde K}\big|,
\end{align*}
so it suffices to show $\E\Big[\big|(\mathcal{U}\setminus \mathcal{K})\cap \mathcal{\widetilde K}\big|\Big] \leq d^{1-\delta}$. We have
\begin{align*}
    (\mathcal{U}\setminus \mathcal{K})\cap \mathcal{\widetilde K} \subseteq \{&c' \in N_H(c) \,:\, \phi(L^{-1}(u)) = \blank,\, c' \in \Happy,\, N_H(c') \cap \Bad \cap \col(A) \neq \emptyset\}. 
\end{align*}
Note that if $c' \in \Happy$, we have
$$\P\big[N_H(c') \cap \Bad \cap \col(A) \neq \emptyset\big] \leq  \frac{\eta}{\ell} \, d^{1-\delta} < d^{-\delta},$$
\noindent
from which it follows
\[
    \E\Big[\big|(\mathcal{U}\setminus \mathcal{K})\cap \mathcal{\widetilde K}\big|\Big] \leq \sum_{\substack{c' \in N_H(c) \\ c' \in \Happy}} \P\big[N_H(c') \cap \Bad \cap \col(A) \neq \emptyset\big] \leq d \, d^{-\delta} = d^{1-\delta}.\qedhere
\]
\end{proof}
We are now ready to finish the proof of Lemma \ref{UminusKConcentration}.

\begin{proof}[Proof of Lemma~\ref{UminusKConcentration}]
    Observe that $(\mathcal{U}\setminus \widetilde{\mathcal{K}}) \setminus (\mathcal{U}\setminus \mathcal{K})$ is the set of colors $c'\in N_H(c)$ which satisfy that $\phi(L^{-1}(c'))  = \blank$, $c'\in \Sad$, and $c' \in \mathcal{K}$. By Claim~\ref{sadBound}, this implies
\[
    |\mathcal{U} \setminus \widetilde{\mathcal{K}}| - |\mathcal{U} \setminus \mathcal{K}|  \leq d^{1-\beta_2}.
\]
Therefore, with probability at least $1-d^{-110}$, we have the following chain of inequalities:
\begin{align*}
    |\mathcal{U}\setminus \mathcal{K}| &\geq |\mathcal{U}\setminus \widetilde{\mathcal{K}}| - d^{1-\beta_2} \\
    &\geq \widetilde Z - d^{1-\beta_2} &\text{ (since $\widetilde{Z} \subseteq \mathcal{U}\setminus \widetilde{\mathcal{K}}$ )}\\
    &\geq \E[\widetilde Z] - d^{1-\delta_2}-d^{1-\beta_2} &\text{ (by Claim \ref{ZBound})}\\
    &\geq \E\big[|\mathcal{U}\setminus \widetilde{\mathcal{K}}|\big] - d^{1-\delta_2}  - 2d^{1-\beta_2} &\text{ (by \eqref{tildeZ})} \\
    &\geq \E\big[|\mathcal{U}\setminus \mathcal{K}|\big] - d^{1-\delta} - d^{1-\delta_2} - 2d^{1-\beta_2}. &\text{ (by Claim \ref{expectedValuesAreClose})}.
\end{align*}
Since $\beta_1 = 1/(20t)$, we have $\beta_1 \leq \frac{1}{2}\min\{\delta, \delta_2, \beta_2, 1\}$, thus $d^{1-\beta_1} \geq d^{1-\delta} + d^{1-\delta_2} + 2d^{1-\beta_2}$ for $d$ large enough. Therefore,
\begin{align*}
    \P\Big[|\mathcal{U}\setminus \mathcal{K}| \geq \E\big[|\mathcal{U}\setminus \mathcal{K}|\big] - d^{1-\beta_1}\Big] \geq 1-d^{-110},
\end{align*}
as desired.
\end{proof}
We can now complete the proof of Lemma \ref{degreeConcentration}:
\begin{align*}
    & \P\big[d'(c) \geq \E[d'(c)] - \frac{d}{\ell} + d^{1-\beta}\big]\\
    \leq ~& \P\big[d'(c) \geq \E[d'(c)]  + 2d^{1-\beta_1}\big] &\text{(for $d$ large enough)}\\
    \leq ~ & \P\Big[|\mathcal{U}| > \E\big[|\mathcal{U}|\big]+ d^{1-\beta_1}\Big] + \P\Big[|\mathcal{U}\setminus \mathcal{K}| < \E\big[|\mathcal{U}\setminus \mathcal{K}|\big] - d^{1-\beta_1}\Big] &\text{ (union bound) }\\
    \leq ~& d^{-110} + d^{-110} &\text{ (by Lemmas \ref{uncolorConcentration} and \ref{UminusKConcentration})}\\
    \leq ~ & d^{-100}.
\end{align*}

\section{Proof of Theorem \ref{mainTheorem}}\label{sectionIterations}

In this section we prove Theorem \ref{mainTheorem} by iteratively applying Lemma \ref{iterationTheorem} until we reach a stage where we can apply Proposition \ref{finalBlow}. To do so, we first define the parameters for the graph and cover at each iteration and then define $d_0$ such that the graphs at each iteration will satisfy the conditions of Lemma \ref{iterationTheorem}. This section follows similarly to \cite[Chapter 12]{MolloyReed}.

We use the notation of Theorem~\ref{mainTheorem}. Let 
\[G_1 \defeq G,\quad \mathcal{H}_1 = (L_1, H_1) \defeq \mathcal{H},\quad \ell_1 \defeq (1+\eps)d/\log d,\quad d_1 \defeq d,\] where we may assume that $\eps$ is sufficiently small, say $\eps < 1/100$. Since $d$ is large, we may also assume that $\ell_1$ is an integer by slightly modifying $\eps$ if necessary. Define \[\kappa \defeq (1+\eps/2)\log (1+\eps/100)\approx \eps/100,\] and fix $\eta \defeq \kappa/\log d$, so that $\eta$ is the same each time we apply Lemma~\ref{iterationTheorem}. 
Set $\beta \defeq 1/(25t)$ and recursively define the following parameters for each $i \geq 1$: 
\begin{align*}
    \keep_i &\defeq \left(1 - \frac{\kappa}{\ell_i \log d}\right)^{d_i},
    & \uncolor_i &\defeq 1 - \frac{\kappa}{\log d}\, \keep_i,\\
    \ell_{i+1} &\defeq \left\lceil\keep_i\, \ell_i- \ell_i^{1 - \beta}\right\rceil, & d_{i+1}&\defeq \left\lfloor\keep_i\, \uncolor_i\, d_i + d_i^{1 - \beta}\right\rfloor.
\end{align*} 
Suppose that at the start of iteration $i$, the following numerical conditions hold:
\begin{enumerate}[label=\ep{\normalfont\arabic*}]
    \item\label{item:1} $d_i \geq \tilde{d}$,
    \item\label{item:2} $\eta\,d_i < \ell_i < 8d_i$,
    \item\label{item:3} $s \leq d_i^{1/4}$,
    \item\label{item:4} $t \leq \dfrac{\tilde{\alpha}\log d_i}{\log\log d_i}$,
    \item\label{item:5} $\dfrac{1}{\log^5d_i} < \eta < \dfrac{1}{\log d_i}$.
\end{enumerate}
Furthermore, suppose that we have a graph $G_i$ and a DP-cover $\mathcal{H}_i = (L_i, H_i)$ of $G_i$ such that:
\begin{enumerate}[resume,label=\ep{\normalfont\arabic*}]
    \item\label{item:6} $H_i$ is $K_{s,t}$-free,
    \item $\Delta(H_i) \leq d_i$,
    \item\label{item:8} $|L_i(v)| \geq \ell_i$ for all $v \in V(G_i)$.
\end{enumerate}
Then we may apply Lemma~\ref{iterationTheorem} to obtain a partial $\mathcal{H}_i$-coloring $\phi_i$ of $G_i$ and an assignment of subsets $L_{i+1}(v) \subseteq (L_i)_{\phi_i}(v)$ to each vertex $v \in V(G_i) \setminus \dom(\phi_i)$ such that, setting
\[
    G_{i+1} \defeq G_i[V(G_i) \setminus \dom(\phi_i)] \quad \text{and} \quad H_{i+1} \defeq H_i \left[\bigcup_{v \in V(G_{i+1})} L_{i+1}(v)\right],
\]
we get that conditions \ref{item:6}--\ref{item:8} hold with $i+1$ in place of $i$. 
Note that, assuming $d_0$ is large enough and $\alpha$ is small enough, conditions \ref{item:1}--\ref{item:8} are satisfied initially \ep{i.e., for $i = 1$}. Our goal is to show that there is some value $i^\star \in \N$ such that:
\begin{itemize}
    \item for all $1 \leq i < i^\star$, conditions \ref{item:1}--\ref{item:5} hold, and
    \item we have $\ell_{i^\star} \geq 8d_{i^\star}$.
\end{itemize}
Since conditions \ref{item:6}--\ref{item:8} hold by construction, we will then be able to iteratively apply Lemma~\ref{iterationTheorem} $i^\star - 1$ times and then complete the coloring using Proposition~\ref{finalBlow}.


We first show that the ratio $d_i/\ell_i$ is decreasing for $d_i$, $\ell_i$ large enough.

\begin{Lemma}\label{decreasingSeq}
    Suppose that for all $j \leq i$, we have 
    $\ell_j^\beta$, $d_j^\beta \geq 30\log^2 d$ and $\ell_j \leq 8d_j$. Then
    \[
        \frac{d_{i+1}}{\ell_{i+1}} \,\leq\, \frac{d_i}{\ell_i}.
    \]
\end{Lemma}
\begin{proof}
The proof is by induction on $i$. Assume the statement holds for all values less than $i$. In particular, $d_i/\ell_i \leq d_1/\ell_1 < \log d$. Using this we find the following bound: 
\begin{align*}
    \keep_i\, \uncolor_i &= \keep_i\,\left(1 - \frac{\kappa}{\log d}\, \keep_i\right) \\
    &= \keep_i - \frac{\kappa}{\log d}\, \left(1 - \frac{\kappa}{\ell_i \log d}\right)^{2d_i} \\
    &\leq \keep_i - \frac{\kappa}{\log d}\,\left(1 - \frac{2\kappa d_i}{\ell_i\log d}\right) \\
    &\leq \keep_i - \frac{\kappa}{2\log d} \\
    &\leq \keep_i - 3\ell_i^{-\beta}.
\end{align*}
With this computation in mind, we have:
\begin{align*}
    \frac{d_{i+1}}{\ell_{i+1}} &\leq \frac{\keep_i\, \uncolor_i\, d_i + d_i^{1-\beta}}{\keep_i\, \ell_i - \ell_i^{1-\beta}} \\
    &\leq \frac{d_i\,(\keep_i - 3\ell_i^{-\beta} + d_i^{-\beta})}{\ell_i\,(\keep_i - \ell_i^{-\beta})} \\
    &\leq \frac{d_i}{\ell_i}.
\end{align*}
The last inequality follows since $\ell_i \leq 8d_i$ and $8^\beta \leq 2$ for $d$ large enough.
\end{proof}

For computational purposes, it is convenient to remove the error terms $\ell_i^{1-\beta}$ and $d_i^{1-\beta}$ from the definitions of $d_{i+1}$ and $\ell_{i+1}$. This is done in the following lemma. 

\begin{Lemma}\label{boundHatDiff}
Let $\hat{\ell}_1 \defeq \ell_1$, $\hat{d}_1 \defeq d_1$, and recursively define:
\begin{align*}
\hat{\ell}_{i+1} &\defeq \keep_i\, \hat{\ell}_i, \\ 
\hat{d}_{i+1} &\defeq \keep_i\, \uncolor_i\, \hat{d}_i.
\end{align*}
If for all $1\leq j < i$, we have $d_j^\beta,\, \ell_j^\beta \geq 30\log^4 d$ and $\ell_j \leq 8d_j$,  then
\begin{itemize}
    \item $|\ell_i - \hat{\ell}_i| \leq \hat{\ell}_i^{1-\beta/2}$,
    \item $|d_i - \hat{d}_i| \leq \hat{d}_i^{1-\beta/2}$.
\end{itemize}
\end{Lemma}
\begin{proof}
Before we proceed with the proofs, let us record a few inequalities. By Lemma \ref{decreasingSeq}, 
\begin{equation}\label{eq:keep_lower}
    \keep_i \geq 1 - \frac{d_i\, \kappa}{\ell_i\log d} \geq 1-\kappa.
\end{equation}
Also, assuming $d$ is large enough, we have 
\begin{equation}\label{eq:keep_upper}
    \keep_i \leq \exp\left(-\dfrac{\kappa\, d_i}{\ell_i\log d}\right) \leq \exp\left(-\dfrac{\kappa}{8\log d}\right) \leq 1 - \frac{\kappa}{10\log d}.
\end{equation}
It follows from \eqref{eq:keep_lower} that
$$\keep_i\, \uncolor_i = \keep_i - \frac{\kappa}{\log d}\, \keep_i^2 \geq 1 - \kappa \left( 1 + \frac{\keep_i^2}{\log d}\right) \geq 1 - 2\kappa.$$
Since $\kappa < 1/4$, the function $f(x) = x^{1-\beta/2} - x$ is decreasing on $[1-2\kappa,1]$. It follows from \eqref{eq:keep_upper} that
\begin{equation}\label{eq:1-omega/21}
    \keep_i^{1-\beta/2} - \keep_i \geq \left(1 - (1-\beta/2)\,\frac{\kappa}{10\log d}\right) - \left(1 - \frac{\kappa}{10\log d}\right) = \frac{\beta \kappa}{20\log d}.
\end{equation}
Also, we can write
\begin{equation}\label{eq:1-omega/22}
    (\keep_i\, \uncolor_i)^{1-\beta/2} - \keep_i\, \uncolor_i \geq \keep_i^{1-\beta/2} - \keep_i \geq \frac{\beta \kappa}{20\log d}.
\end{equation}
Now we are ready to prove Lemma~\ref{boundHatDiff} by induction on $i$. 
Note that $\hat{\ell}_i \geq \ell_i$, $\hat{d}_i \leq d_i$. For the base case $i = 1$, the claim is trivial. Assume now that it holds for some $i$ and consider $i + 1$. We have
\begin{align*}
    \hat{\ell}_{i+1} &= \keep_i\, \hat{\ell}_i \\
    &\leq \keep_i\,(\ell_i + \hat{\ell}_i^{1 - \beta/2}) &\text{(by the inductive hypothesis)} \\
    &\leq \ell_{i+1} +  \ell_i^{1 - \beta} + \left(\keep_i^{1-\beta/2} - \frac{\beta \kappa}{20\log d}\right)\, \hat{\ell}_{i}^{1-\beta/2} &\text{(by \eqref{eq:1-omega/21})} \\
    &= \ell_{i+1} + \hat{\ell}_{i+1}^{1-\beta/2} + \ell_i^{1 - \beta} - \frac{\beta\kappa}{20\log d}\,\hat{\ell}_i^{1-\beta/2}.
\end{align*}
\noindent
It remains to argue that
$$\frac{\beta\kappa}{20\log d}\,\hat{\ell}_i^{1-\beta/2} \geq \ell_i^{1-\beta},$$
which is equivalent to
$$\frac{\ell_i^{1-\beta}}{\hat{\ell}_i^{1-\beta/2}} \leq \frac{\beta\kappa}{20\log d}.$$
To this end, we write
$$\frac{\ell_i^{1-\beta}}{\hat{\ell}_i^{1-\beta/2}} \leq \frac{\ell_i^{1-\beta}}{\ell_i^{1-\beta/2}} = \ell_i^{-\beta/2} \leq \frac{1}{5\log^2d} < \frac{\beta\kappa}{20\log d},$$
since $\beta = \Omega(\log \log d/\log d)$. Thus, the claim holds for $d$ large enough.

The argument for $\hat{d}_{i+1}$ is almost identical. We have
\begin{align*}
    \hat{d}_{i+1} &= \keep_i\, \uncolor_i\, \hat{d}_i \\
    &\geq \keep_i\, \uncolor_i\,(d_i - \hat{d}_i^{1 - \beta/2}) &\text{(by the inductive hypothesis)}\\
    &\geq d_{i+1} - d_i^{1-\beta } - \left((\keep_i\,\uncolor_i)^{1-\beta /2} - \frac{\beta\kappa}{20\log d}\right)\,\hat{d}_i^{1-\beta/2} &\text{(by \eqref{eq:1-omega/22})} \\
    &= d_{i+1} - \hat{d}_{i+1}^{1-\beta/2} - d_i^{1-\beta} + \frac{\beta\kappa}{20\log d}\,\hat{d}_i^{1-\beta/2}.
\end{align*}
\noindent
It remains to argue that
$$\frac{\beta\kappa}{20\log d}\,\hat{d}_i^{1-\beta/2} \geq d_i^{1-\beta},$$
which is equivalent to
$$\frac{d_i^{1-\beta}}{\hat{d}_i^{1-\beta/2}} \leq \frac{\beta\kappa}{20\log d}.$$
To this end, we write
$$\frac{d_i^{1-\beta}}{\hat{d}_i^{1-\beta/2}} \leq \frac{d_i^{-\beta}}{\hat{d}_i^{1-\beta/2}}\,\left(\hat{d}_i + \hat{d}_i^{1-\beta/2}\right) = d_i^{-\beta}(\hat{d}_i^{\beta /2}+1) \leq 2d_i^{-\beta /2} \leq \frac{1}{2\log^2d} < \frac{\beta\kappa}{20\log d},$$
and so, the claim holds for $d$ large enough.
\end{proof}


Next we show that $\ell_i$ never gets too small:

\begin{Lemma}\label{lemma:l_lower}
    Suppose that for all $j < i$, we have $\ell_j \leq 8d_j$. Then $\ell_i \geq d^{\eps/15}$.
\end{Lemma}
\begin{proof}
    For brevity, set $r_i \defeq d_i/\ell_i$ and $\hat{r}_i \defeq \hat{d}_i/\hat{\ell}_i$. The proof is by induction on $i$. The base case $i = 1$ is clear. Now we assume that the desired bound holds for $\ell_1$, \ldots, $\ell_i$ and consider $\ell_{i+1}$. 
    Assuming $d$ is large enough, we have
\begin{equation}\label{eq:bound1}
    1 - \frac{\kappa}{\ell_i\log d} \geq \exp{\left(-\frac{\kappa}{(1-\eps/4)\ell_i\log d}\right)}.
\end{equation}
Note that $r_1 = \hat{r}_1 = \log d/(1+\eps)$ and, assuming $\eps < 1/100$, $(1-\eps/4)(1+\eps) \geq (1 +\eps/2)$. Hence,
\begin{align*}
    \keep_i &= \left(1 - \frac{\kappa}{\ell_i \log d}\right)^{d_i} \\
    &\geq \exp{\left(-\frac{\kappa}{(1-\eps/4)\log d}\, r_i\right)} &\text{(by \eqref{eq:bound1})}\\
    &\geq \exp{\left(-\frac{\kappa}{(1-\eps/4)\log d}\, r_1\right)} &\text{(by Lemma \ref{decreasingSeq})}\\
    &\geq \exp{\left(-\frac{\kappa}{(1+\eps/2)}\right)}. 
\end{align*}
\noindent
With this bound on $\keep_i$, we can bound $\hat{r}_i$ as follows:
\begin{align*}
    \hat{r}_i &= \hat{r}_1\prod\limits_{j < i}\uncolor_j \\
    &= \hat{r}_1\prod\limits_{j < i}\left(1 - \frac{\kappa}{\log d}\, \keep_j \right) \\
    &\leq \frac{\log d}{1+\eps}\left(1 - \frac{\kappa}{\log d}\, \exp{\left(-\frac{\kappa}{(1+\eps/2)}\right)}\right)^{i-1}.
\end{align*}
\noindent
Applying Lemma \ref{boundHatDiff}, we get a bound on $r_i$ for $d$ large enough in terms of $\eps$: 
\begin{align*}
    r_i &\leq \hat{r}_i\left(\frac{1 + \hat{d}_i^{-\beta/2}}{1 - \hat{\ell}_i^{-\beta/2}}\right) \\
    &\leq \hat{r}_i(1 + \hat{\ell}_i^{-\beta/2} + \hat{d}_i^{-\beta/2}) \\
    &\leq \hat{r}_i\left(1 + O(d^{-\eps\beta/30})\right)\\
    &< \frac{\log d}{1+\eps/2}\left(1 - \frac{\kappa}{\log d}\, \exp{\left(-\frac{\kappa}{(1+\eps/2)}\right)}\right)^{i-1}.
\end{align*}
Note that for $\eps$ small enough, $(1-\eps/4)(1+\eps/2) \geq 1+\eps/8$. Applying this and the above bound on $r_i$, we can get a better bound on $\keep_i$:
\begin{align*}
    \keep_i &\geq \exp{\left(-\frac{\kappa}{(1-\eps/4)\log d}\, r_i\right)} \\
    &\geq \exp{\left(-\frac{\kappa}{(1-\eps/4)\log d}\, \frac{\log d}{1+\eps/2}\left(1 - \frac{\kappa}{\log d}\, \exp{\left(-\frac{\kappa}{(1+\eps/2)}\right)}\right)^{i-1}\right)} \\
    &\geq \exp{\left(-\frac{\kappa}{(1+\eps/8)}\left(1 - \frac{\kappa}{\log d}\, \exp{\left(-\frac{\kappa}{(1+\eps/2)}\right)}\right)^{i-1}\right)}.
\end{align*}
With this bound on $\keep_i$, we can get a lower bound on $\hat{\ell}_{i+1}$ as follows:
\begin{align*}
    \hat{\ell}_{i+1} &= \hat{\ell}_1\,\prod\limits_{j \leq i}\keep_j \\
    &\geq \hat{\ell}_1\,\prod\limits_{j \leq i}\exp{\left(-\frac{\kappa}{(1+\eps/8)}\left(1 - \frac{\kappa}{\log d}\, \exp{\left(-\frac{\kappa}{(1+\eps/2)}\right)}\right)^{j-1}\right)} \\
    &= (1+\eps)\,\frac{d}{\log d}\, \exp{\left(-\frac{\kappa}{(1+\eps/8)}\sum\limits_{j \leq i}\left(1 - \frac{\kappa}{\log d}\, \exp{\left(-\frac{\kappa}{(1+\eps/2)}\right)}\right)^{j-1}\right)} \\
    &\geq (1+\eps)\,\frac{d}{\log d}\, \exp{\left(-\frac{\kappa}{(1+\eps/8)}\sum\limits_{j =1}^\infty\left(1 - \frac{\kappa}{\log d}\, \exp{\left(-\frac{\kappa}{(1+\eps/2)}\right)}\right)^{j-1}\right)} \\
    &= (1+\eps)\,\frac{d}{\log d}\, \exp{\left(-\frac{\log d}{(1+\eps/8)}\, \exp{\left(\frac{\kappa}{(1+\eps/2)}\right)}\right)} \\
    &= (1+\eps)\,\frac{d}{\log d}\, d^{\left(-\dfrac{\exp{\left(\kappa/(1+\eps/2)\right)}}{(1+\eps/8)}\right)}.
\end{align*}
Recalling that $\kappa = (1+\eps/2)\log (1+\eps/100)$, we get
$$\frac{\exp{\left(\kappa/(1+\eps/2)\right)}}{(1+\eps/8)} = \frac{1+\eps/100}{1+\eps/8} < 1-\eps/9.$$
Therefore, for $d$ large enough, we get
$$\hat{\ell}_{i+1} > (1+\eps)\,\frac{d}{\log d}\, d^{\eps/9-1} > d^{\eps/10}.$$
Applying Lemma \ref{boundHatDiff}, we finally get the bound we desire:
\[
    \ell_{i+1} \geq \hat{\ell}_{i+1} - \hat{\ell}_{i+1}^{1-\beta/2} \geq d^{\eps/10}(1 - \hat{\ell}_{i+1}^{-\beta/2}) \geq d^{\eps/15}. \qedhere
\]
\end{proof}

We can now finally establish the existence of the desired bound $i^\star$:

\begin{Lemma}\label{boundListSize}
There exists an integer $i^\star \geq 1$ such that $\ell_{i^\star} \geq 8d_{i^\star}$.
\end{Lemma}
\begin{proof}
%
%
As in the proof of Lemma~\ref{lemma:l_lower}, set $r_i \defeq d_i/\ell_i$ and $\hat{r}_i \defeq \hat{d}_i/\hat{\ell}_i$. Suppose, toward a contradiction, that $\ell_i < 8d_i$ \ep{i.e., $r_i > 1/8$} for all $i \geq 1$. By Lemma~\ref{lemma:l_lower}, this implies that $\ell_i \geq d^{\eps/15}$ for all $i$. Note that $\hat{r}_i = \uncolor_i\, \hat{r}_{i-1}$ is a decreasing sequence. Furthermore, $\keep_j \geq \keep_1 \geq 1-\frac{\kappa}{1+\eps} \geq 1/2$. Thus,
\begin{align*}
    r_{i} &\leq 2\hat{r}_{i} \\
    &\leq 2\hat{r}_1\prod\limits_{j < i}\left(1 - \frac{\kappa}{\log d}\, \keep_j\right) \\
    &\leq 2\hat{r}_1\left(1 - \frac{\kappa}{2\log d}\right)^i \\
    & \leq 2\log d\, \exp{\left(-\frac{\kappa}{2\log d}\, i\right)}.
\end{align*}
For $i \geq \frac{10}{\kappa}\,\log d\log\log d$, the last expression is less than $1/8$; a contradiction.
\end{proof}

Let $i^\star \geq 1$ be the smallest integer such that $\ell_{i^\star} \geq 8d_{i^\star}$ (which exists by Lemma~\ref{boundListSize}). Take any $i < i^\star$. We need to verify conditions \ref{item:1}--\ref{item:5}. Note that Lemma~\ref{lemma:l_lower} yields
\begin{equation}\label{eq:landd}
    \ell_{i} \geq d^{\eps/15} \quad \text{and} \quad d_{i} \geq \frac{\ell_{i}}{8} \geq \frac{d^{\eps/15}}{8} \geq d^{\eps/20}.
\end{equation}
Therefore, condition \ref{item:1} holds assuming that $d_0 > \tilde{d}^{20/\eps}$. For \ref{item:2}, we use Lemma \ref{decreasingSeq} to write
\[
    \frac{\ell_i}{d_i} \geq \frac{\ell_1}{d_1} \geq \frac{1}{\log d} \geq \eta.
\]
Next, due to \eqref{eq:landd}, we can take $\alpha$ so small that
\[
    s \leq d^{\alpha\eps} \leq d_i^\frac{1}{4} \quad \text{and} \quad t \leq \frac{\alpha\eps\log d}{\log \log d} \leq \frac{\tilde{\alpha} \log d_i}{\log \log d_i},
\]
which yields conditions \ref{item:3} and \ref{item:4}. Finally, it follows for $d$ large enough that
\[
    \frac{1}{\log^5d_i} \leq \frac{1}{(\eps/20)^5\log^5 d} \leq \eta \leq \frac{1}{\log d} \leq \frac{1}{\log d_i},
\]
so \ref{item:5} holds as well. As discussed earlier, we can now iteratively apply Lemma~\ref{iterationTheorem} $i^\star - 1$ times and then complete the coloring using Proposition~\ref{finalBlow}. This completes the proof of Theorem~\ref{mainTheorem}. 



    \subsection*{Acknowledgements}
        We are grateful to the anonymous referee for helpful suggestions.

\printbibliography

@ARTICLE{JMTheorem,
	AUTHOR = "Bernshteyn, A.",
	TITLE = "{The Johansson-Molloy theorem for DP-coloring}",
	JOURNAL = "Rand. Struct. Algor.",
	YEAR = "2019",
	volume = {54},
	pages = {653--664},
}

@BOOK{MolloyReed,
    AUTHOR = "Molloy, M. and Reed, B.",
    TITLE = "{Graph Colourings and the Probabilistic Method}",
    PUBLISHER = "Springer",
    YEAR = "2002",
}

@ARTICLE{GraphEmbedding,
	AUTHOR = "Bonamy, M. and Perrett, T. and Postle, L.",
	TITLE = "{Colouring Graphs with Sparse Neighbourhoods: Bounds and Applications}",
	archivePrefix = "arXiv",
	note = {arXiv:1810.06704},
	YEAR = "2018",
	MONTH = "oct",
}

@ARTICLE{DPCol,
	AUTHOR = "Dv\v{o}r\'ak, Z. and Postle, L.",
	TITLE = "{Correspondence coloring and its application to list-coloring planar graphs without cycles of lengths 4 to 8}",
	JOURNAL = "J. Combin. Theory",
	series = {B},
% 	archivePrefix = "arXiv",
% 	note = {arXiv:1508.03437},
	YEAR = "2018",
	volume = {129},
	pages = {38--54},
	%MONTH = "mar",
}

@unpublished{Pallette,
	AUTHOR = "Alon, N. and Assadi, S.",
	TITLE = "{Palette sparsification beyond $(\Delta +1)$ vertex coloring}",
	howpublished = {\url{https://arxiv.org/abs/2006.10456} (preprint)},
	date = "2020",
	%MONTH = "jul",
}

@ARTICLE{GirthRegular1,
	AUTHOR = "Imrich, W.",
	TITLE = "{Explicit construction of regular graphs without small cycles}",
    JOURNAL = "Combinatorica",
	YEAR = "1984",
	volume = {4},
	pages = {53--59},
}

@ARTICLE{GirthRegular2,
	AUTHOR = "Margulis, G. A.",
	TITLE = "{Explicit constructions of graphs without short cycles and low density codes}",
    JOURNAL = "Combinatorica",
	YEAR = "1982",
	volume = {2},
	pages = {71--78},
}

@ARTICLE{KST,
	AUTHOR = "K\H{o}v\`{a}ri, T. and S\'{o}s, V.T. and Tur\'{a}n, P.",
	TITLE = "{On a problem of K. Zarankiewicz}",
	JOURNAL = "Colloquium Mathematicum",
	YEAR = "1954",
	pages = {50--57},
	volume = {3},
}

@ARTICLE{KST2,
	AUTHOR = "Hylt\'en-Cavallius, C.",
	TITLE = "{On a combinatorial problem}",
	JOURNAL = "Colloquium Mathematicum",
	YEAR = "1958",
	pages = {61--65},
	volume = {6},
}

@ARTICLE{ExceptionalTal,
	AUTHOR = "Bruhn, H. and Joos, F.",
	TITLE = "{A stronger bound for the strong chromatic index}",
% 	archivePrefix = "arXiv",
% 	note = {arXiv:1504.02583},
    JOURNAL = "Combin. Probab. Comput.",
	YEAR = "2018",
% 	MONTH = "jan",
    pages = {21--43},
	volume = {27},
}

@ARTICLE{AKSConjecture,
	AUTHOR = "Alon, N. and Krivelevich, M. and Sudakov, B.",
	TITLE = "{Coloring graphs with sparse neighborhoods}",
	JOURNAL = "J. Combin. Theory",
	series = {B},
	YEAR = "1999",
	volume = {77},
	pages = {73--82},
}

@unpublished{Joh_sparse,
    author = {A. Johansson},
    title = {The choice number of sparse graphs},
	howpublished = {\url{https://www.cs.cmu.edu/~anupamg/down/johansson-choice-number-of-sparse-graphs-coloring-kr-free.pdf} (preprint)},
	date = {1996},
}

@article{Kim95,
	author = {J.H. Kim},
	title = {On Brooks' Theorem for sparse graphs},
	journaltitle = {Combin. Probab. Comput.},
	date = {1995},
	volume = {4},
	pages = {97--132},
}

@report{Joh_triangle,
	author = {A. Johansson},
	title = {Asymptotic choice number for triangle free graphs},
	type = {Technical Report 91--95},
	institution = {DIMACS},
	date = {1996},
}

@article{PS15,
	author = {S. Pettie and H.-H. Su},
	title = {Distributed coloring algorithms for triangle-free graphs},
	journaltitle = {Information and Computation},
	date = {2015},
	volume = {243},
	pages = {263--280},
}

@article{Molloy,
	author = {M. Molloy},
	title = {The list chromatic number of graphs with small clique number},
	journaltitle = {J. Combin. Theory},
	series = {B},
	volume = {134},
	pages = {264--284},
	date = {2019},
}

@unpublished{DKPS,
    author = {E. Davies and R.J. Kang and F. Pirot and J.-S. Sereni},
    title = {Graph structure via local occupancy},
	howpublished = {\url{https://arxiv.org/abs/2003.14361} (preprint)},
	date = {2020},
}

@unpublished{CK,
    author = {S. Cambie and R.J. Kang},
    title = {Independent transversals in bipartite correspondence-covers},
	howpublished = {\url{https://arxiv.org/abs/2009.05428} (preprint)},
	date = {2020},
}

@article{AminiReed,
	author = {O. Amini and B. Reed},
	title = {List colouring constants of triangle free graphs},
	journaltitle = {Electron. Notes Discret. Math.},
	volume = {30},
	pages = {135--140},
	date = {2008},
}

@unpublished{Nibble,
    author = {D.Y. Kang and T. Kelly and D. K{\"{u}}hn and A. Methuku and D. Osthus},
    title = {Graph and hypergraph colouring via nibble methods: A survey},
	howpublished = {\url{https://arxiv.org/pdf/2106.13733} (preprint)},
	date = {2021},
}

@unpublished{nbhd,
	author = {J. Anderson and A. Bernshteyn and A. Dhawan},
	title = {Coloring graphs with forbidden almost bipartite subgraphs},
	%date = {2021},
	howpublished = {in preparation},
}

@article{BollobasIndependence,
	author = {B. Bollob{\'{a}}s},
	title = {The independence ratio of regular graphs},
	journaltitle = {Proc. Amer. Math. Soc.},
	date = {1981},
	volume = {83},
	number = {2},
	pages = {433--436},
}

@article{Achlioptas,
	author = {D. Achlioptas and A. Coja-Oghlan},
	title = {Algorithmic barriers from phase transitions},
	journaltitle = {IEEE Symposium on Foundations of Computer Science (FOCS)},
	date = {2008},
	pages = {793--802},
	addendum = {Full version: \url{https://arxiv.org/abs/0803.2122}},
}

@article{Zdeborova,
	author = {L. Zdeborov{\'{a}} and F. Krz{\k{a}}ka{\l}a},
	title = {Phase transitions in the coloring of random graphs},
	journaltitle = {Phys. Rev. E},
	date = {2007},
	volume = {76},
	pages = {031131},
}

@article{RV,
	author = {M. Rahman and B. Vir{\'{a}}g},
	title = {Local algorithms for independent sets are half-optimal},
	journaltitle = {Ann. Probab.},
	volume = {45},
	number = {3},
	pages = {1543--1577},
	date = {2017},
}
\end{document}